\documentclass[11pt]{amsart}

\usepackage[margin=1in]{geometry}
\usepackage{amsmath,amssymb,amsthm,mathtools}
\usepackage{hyperref}
\usepackage{graphicx}
\usepackage{microtype}
\usepackage{xcolor}
\usepackage{enumitem}
\usepackage{booktabs}
\usepackage{listings}
\usepackage{algorithm}
\usepackage{algpseudocode}
\usepackage[most]{tcolorbox}

\newtheorem{conjecture}{Conjecture}

\newtheorem{lemma}{Lemma}
\newtheorem{definition}{Definition}
\newtheorem{remark}{Remark}

\newcommand{\sgn}{\mathrm{sgn}}
\newcommand{\Ln}{\mathcal{L}_n}
\newcommand{\E}{\mathcal{E}}
\newcommand{\Oset}{\mathcal{O}}
\newcommand{\Res}{\mathrm{Res}}

\title{Evolving Local Corrections for Global Constructions in Combinatorics}
\author{Gergely B\'erczi}
\address{Aarhus University}
\email{gergely.berczi@math.au.dk}

\date{}

\begin{document}

\begin{abstract}
Many open problems in combinatorics admit reformulations in which a global construction can be achieved by the repeated application of small, finite correcting steps. This paper presents three computational case studies of this principle, carried out using AlphaEvolve as an experimental engine for proposing and iteratively refining such certificates. The problems we studied are: reconstruction of bipartite and planar graphs from vertex-deleted subgraphs; the Alon-Tarsi parity problem for Latin squares, approached via sign-reversing involutions built from local trades; Rota's Basis Conjecture, studied through local exchange policies on collections of bases. In these three problems the correcting steps take the form of a reconstruction rule, a parity-reversing involution, and a transversal family of bases, respectively.

For each problem, we describe the experimental setup, the scoring protocols, and the outcomes of the searches, leading to concrete conjectures concerning the existence and structure of the relevant correcting steps. The aim is not to claim proofs, but rather to produce explicit algorithms and to reveal structural patterns that appear amenable to subsequent analysis by traditional mathematical methods.
\end{abstract}

\maketitle

\section{Introduction}
\label{sec:introduction}

This paper reports on a series of experiments centered around a common guiding idea: to use the evolutionary search model AlphaEvolve \cite{DeepMindAlphaEvolve2025,AlphaEvolveArxiv} as a construction-driven search engine for exploring difficult constructive problems in combinatorics. These are problems that assert the existence of objects subject to global constraints, but for which brute-force search is computationally infeasible.
The three problems we study are the reconstruction of graphs from their vertex-deleted subgraphs (with an emphasis on bipartite and planar families), the Alon-Tarsi conjecture on the parity imbalance of Latin squares, and Rota's Basis Conjecture on rearranging matroid bases into a basis-by-basis grid. 

Although the Alon-Tarsi conjecture implies the Rota Basis Conjecture in certain cases, at first look these problems belong to distinct areas of combinatorics.  From the viewpoint adopted here, however, they share a common structural feature: each induces an energy landscape on a vast discrete space $\mathcal{C}$ of candidate certificates. Valid certificates correspond to global minima, while the verifier assigns to each $C \in \mathcal{C}$
a vector of nonnegative defect measures that quantify its distance from validity. These defects, such as incorrect degrees or broken column conditions, define local contributions to the energy and naturally suggest local moves that decrease it, thereby organizing the search as a descent process on a highly structured combinatorial landscape.

Across the three projects we found it useful to treat the unknown certificate $C$ as something that should be built by a sequence of small correcting steps,
much as a numerical optimizer builds a minimizer by small descent steps.
Concretely, one fixes a space $\mathcal{C}$ of candidate objects (graphs, Latin squares, basis grids), together with a computable penalty or loss function $\mathcal{L}:\mathcal{C}\to \mathbb{R}_{\ge 0}$ whose zero set is precisely the set of valid certificates, and a family $\mathcal{M}$ of local moves $m:\mathcal{C}\to\mathcal{C}$ that preserve the ambient combinatorial type (for instance degree sequences, Latin constraints) and the loss function decreases at each move, that is $\mathcal{L}(m(C))<\mathcal{L}(C)$ holds for all $m\in \mathcal{M}, C\in \mathcal{C}$. 

The guiding heuristic is that, if a conjecture is true for global reasons, then there should exist a discrete Lyapunov trajectory, that is, 
a local repair dynamics that (typically) drives $\mathcal{L}$ downward and avoids the main traps.
AlphaEvolve is then used not to search directly for a single certificate $C$,
but to search for python scripts implementing a good repair dynamics, that is, a policy that repeatedly proposes moves in $\mathcal{M}$
that reduce $\mathcal{L}$.
This principle does not replace proof. Rather, it serves as a generator of explicit hypotheses such as which local moves seem sufficient and which deterministic tie-breaks prevent cycling.

The rest of this introduction presents the statements of the three conjectures we study in this paper, with more detailed background in the corresponding sections. 

\subsection{Graph Reconstruction Conjecture}
\label{sec:intro-graph}

Let $G$ be a finite simple undirected graph with vertex set $V(G)$ of size $n$.
For $v\in V(G)$ we write $G-v$ for the induced subgraph obtained by deleting $v$ and all incident edges.
The \emph{deck} of $G$ is the multiset of vertex-deleted subgraphs
\[
\mathcal{D}(G) = \{ G-v : v\in V(G)\}.
\]
Two graphs $G,H$ on $n$ vertices have the same deck if their decks agree as multisets of unlabeled graphs.
A graph $G$ is \emph{reconstructible} if every graph with the same deck is isomorphic to $G$.

\begin{conjecture}[Graph Reconstruction Conjecture (Ulam-Kelly) \cite{Ulam1960,Kelly1957}]
Every finite simple graph on at least three vertices is reconstructible from its deck.
\end{conjecture}

The conjecture is classical and has generated a large literature; survey-level entry points include \cite{BondyHemminger1977}.
Many special families are known to be reconstructible (for example trees and regular graphs), but no general proof is known.
At the opposite extreme, calculations have verified reconstructibility for all graphs up to modest size; see McKay's tabulations and notes \cite{McKaySmallReconstruction2021}. The deck forgets the labeling of vertices, so the challenge is not merely to recover an adjacency relation but to recover it \emph{up to isomorphism} from unlabeled local views.

The experiments in this paper focus on two regimes in which the deck retains nontrivial combinatorial rigidity while still admitting explicit computational checks. First, we study \emph{bipartite matrix reconstruction}.
A bipartite graph $G=(U\sqcup V,E)$ with $|U|=u$ and $|V|=v$ can be encoded by its $0-1$ incidence matrix $A\in\{0,1\}^{u\times v}$.
Deleting a vertex in $U$ or $V$ corresponds to deleting a row or a column, so the deck becomes a multiset of $(u-1)\times v$ and $u\times (v-1)$ sub-matrices.
This is a concrete reconstruction model in which the unlabeled aspect can be simulated by independently permuting rows and columns in each card.

After this we study reconstructibility of planar graphs. A basic structural question is whether planarity of graphs is itself determined by the deck and, more broadly, which planar graph subclasses are reconstructible or even recognizable. These issues are studied in substantial detail in \cite{BilinskiKwonYu2007}.
In this paper we study planar graphs of minimum degree at least $3$ that are not maximal planar.

\subsection{The Alon-Tarsi conjecture and the parity of Latin squares}
\label{sec:intro-alon}
A Latin square of order $n$ is an $n\times n$ matrix $L=(L_{r,c})_{r,c\in[n]}$ with entries in $[n]=\{0,1,\dots,n-1\}$
such that each symbol appears exactly once in each row and each column. Equivalently, each row and column is a permutation of $[n]$.
For a row $r$, define the permutation $\rho_r\in S_n$ in one-line notation by
\[
\rho_r(c)=L_{r,c}\qquad(c\in[n]),
\]
and for a column $c$ define $\kappa_c\in S_n$ by
\[
\kappa_c(r)=L_{r,c}\qquad(r\in[n]).
\]
\begin{definition}[Parity of Latin squares]
For a Latin square $L$ of order $n$, define
\[
\sgn(L):=\prod_{r=0}^{n-1}\sgn(\rho_r)\cdot \prod_{c=0}^{n-1}\sgn(\kappa_c)\in\{+1,-1\}.
\]
Call $L$ even if $\sgn(L)=+1$ and odd if $\sgn(L)=-1$.
\end{definition}

Let $\E_n$ (resp.\ $\Oset_n$) denote the set of even (resp.\ odd) Latin squares of order $n$ and write
$E_n=\lvert\E_n\rvert$, $O_n=\lvert\Oset_n\rvert$.

\begin{conjecture}[Alon-Tarsi]
If $n$ is even, then
\[
E_n-O_n\neq 0,
\qquad\text{equivalently}\qquad
\sum_{L\in\Ln}\sgn(L)\neq 0.
\]
\end{conjecture}

A standard algebraic encoding explains why this is natural. Let $X=(x_{i,j})$ be an $n\times n$ matrix.
Expanding $\det(X)^n$ produces monomials that correspond to Latin squares, and the coefficient of
$\prod_{i,j}x_{i,j}$ is precisely $\sum_{L\in\Ln}\sgn(L)$. Thus the conjecture is equivalent to the nonvanishing of the coefficient of $\prod_{i,j}x_{i,j}$. This viewpoint connects naturally to polynomial-method tools such as the Combinatorial Nullstellensatz \cite{Alon1999}.

For odd $n$, there is a straightforward sign-reversing involution that forces $E_n=O_n$.
Indeed, swapping two rows (or two columns) is an involution on $\Ln$.
Under the sign convention above, swapping two \emph{rows} multiplies each column permutation by a transposition, so the product of column signs picks up a factor $(-1)^n$; for odd $n$ this is $-1$ and hence $\sgn$ flips. Consequently, for odd $n$ one obtains a sign-reversing involution on all of $\Ln$ and thus $E_n=O_n$.
For even $n$, the same row/column swap is sign-preserving, and no such trivial cancellation is available.

The parity problem for Latin squares is part of a larger circle of ideas in which an algebraic identity controls the existence of a combinatorial object. In their foundational work on the method now called the \emph{Alon-Tarsi method}, Alon and Tarsi related signed enumerations to coloring and orientation problems in graphs \cite{AT92}. Several related determinantal and hyperdeterminantal formulations have also been developed; see for instance Zappa's analysis of Cayley-type determinants in this context \cite{Zappa1997}.

The full Alon-Tarsi conjecture remains open, but nontrivial cases are known.
Drisko proved the conjecture for $n=p+1$ when $p$ is an odd prime \cite{Dr97}, and Glynn proved it for $n=p-1$ when $p$ is an odd prime \cite{Gly10}.
Further extensions and related parity questions are surveyed in the works such as \cite{Sto12,FriedmanMcGuinness2019}.
For odd $n$ the global parity difference vanishes by an elementary row or column swap (as we have seen above), but fixed-diagonal and other constrained variants exhibit subtler behaviour; see, for example, Drisko's fixed-diagonal formulation \cite{Dr98} and more recent asymptotic parity results \cite{CavenaghWanless2016}.

If one can construct a map $F:\Ln\to\Ln$ such that $F$ is an involution, that is, $F(F(L))=L$), $F(L)\neq L$ for most $L$ (so it pairs objects), and $\sgn(F(L))=-\sgn(L)$ for those paired $L$, then the signed sum collapses to the signed sum over the \emph{residual set}
\[
\Res:=\{L\in\Ln: \sgn(F(L))=\sgn(L)\ \text{or $F$ fails to pair $L$ in some way}\}.
\]
In that case
\[
\sum_{L\in\Ln}\sgn(L)=\sum_{L\in\Res}\sgn(L),
\]
so any proof that $\sum_{L\in\Res}\sgn(L)\neq 0$ would establish the conjecture.
The experiments below are best understood as algorithmic attempts to realize this philosophy: evolve a \emph{canonical} local trade that flips sign on most inputs and leaves a small residual with dominant sign. 

\subsection{Rota Basis Conjecture}
Let $M$ be a matroid of rank $n$ on ground set $E$.
Let $B_1,\dots,B_n$ be \emph{pairwise disjoint} bases of $M$.
A \emph{transversal} of $(B_1,\dots,B_n)$ is a set $T\subseteq \bigcup_i B_i$ that meets each $B_i$ in exactly one element.
Equivalently, it is the image of a system of distinct representatives for the partition $(B_1,\dots,B_n)$.

\begin{conjecture}[Rota Basis Conjecture (matroid form)]
\label{conj:RBC}
Given $n$ disjoint bases $B_1,\dots,B_n$ in a rank-$n$ matroid $M$, there exist $n$ pairwise disjoint transversals $T_1,\dots,T_n$ of $(B_1,\dots,B_n)$ such that each $T_j$ is a basis of $M$.
\end{conjecture}

In the linear case, one may take $M$ to be the matroid of column-independence of vectors in an $n$-dimensional vector space $V$.
Then each $B_i$ is a basis of $V$, and a transversal basis is a rainbow basis containing exactly one vector from each $B_i$.
The conjecture asserts that one can partition the multiset $\bigsqcup_i B_i$ into $n$ disjoint rainbow bases.

Rota formulated this conjecture in the late 1980s, and it has since become a central open problem in matroid theory.
Already in early work one sees strong links to determinant identities and to Latin-square phenomena; for example, Onn exhibited a colorful determinantal identity that ties the conjecture to Latin squares \cite{Onn1997}, and Huang-Rota connected the conjecture to Latin-square straightening coefficients \cite{HuangRota1994}.
Chow later reduced Rota's conjecture to a problem involving only three bases \cite{Chow2009}, and Wild proved the conjecture for certain structured instances \cite{Wild1994}.
These connections explain why Rota's conjecture and the Alon-Tarsi conjecture are often discussed in the same breath, and why parity-type information can be unexpectedly relevant (cf.\ \cite{Gly10}).

A number of matroid classes are known to satisfy the conjecture or strong variants.
For example, Geelen and Humphries proved Rota's conjecture for paving matroids \cite{GeelenHumphries2006}, and Geelen-Webb obtained systematic packing and exchange results that guarantee many disjoint transversal bases under additional hypotheses \cite{GeelenWebb2007}.
In the linear setting, substantial quantitative progress has been made in recent years: Buci\'c, Kwan, Pokrovskiy, and Sudakov proved a result guaranteeing about $n/2$ disjoint rainbow bases \cite{BucicKwanPokrovskiySudakov2020}, Sauermann established the conjecture for random bases in a natural probabilistic model \cite{Sauermann2024}, and Pokrovskiy proved an asymptotic form of the conjecture \cite{Pokrovskiy2025}.
Montgomery and Sauermann obtained asymptotically tight packing and covering bounds in related rainbow-base problems \cite{MontgomerySauermann2025}.
Algorithmic and computer-assisted approaches have also entered the picture; for example, Kirchweger, Scheucher, and Szeider used SAT-based methods to verify Rota-type statements in small ranks \cite{KirchwegerScheucherSzeider2022}.

Our experiments focus on a small but instructive linear regime rank $\le 13$ over $\mathbb{F}_2$. The goal is not to resolve the conjecture, but to use scored local-exchange search to find explicit policies that reliably repair invalid columns and reach a full set of transversal bases on adversarially structured test benchmarks.

\subsection{How the experiments were run}

This work follows an experimental rather than a traditional theorem-proof methodology. 
For each of the three problems, the test sets were constructed in advance using standard mathematical judgement and guidance from the literature, and were then fixed as part of the AlphaEvolve evaluation harness. We tried to include difficult cases in the test benchmarks: bipartite and planar graphs with hard reconstruction problems, Latin squares with nontrivial parity structure, and collections of bases close to known obstructions for Rota’s Basis Conjecture. The goal was to bias the search toward methods that handle a broad range of challenging cases, rather than heuristics tuned to typical examples.
As a geometer rather than a specialist in these areas, I approached these problems somewhat outside my primary field of expertise, and the topics therefore lie partially outside my usual comfort zone.

For each problem, approximately ten independent experiments were run, with durations ranging from a few hours to two days. Running multiple shorter experiments, rather than a single long run, proved essential for iteratively improving the test sets and scoring system, closing unintended loopholes, and adjusting the search strategy in response to observed behavior.

AlphaEvolve includes a feature that produces a Gemini-generated AI summary of an experimental run. Our runs typically last 15–24 hours and generate an evolutionary trace of more than $5000$ candidate results and algorithms. The summary distills this trace into a compact narrative of the search dynamics: the main stages of progress, key turning points and breakthroughs where genuinely new ideas appear, and the detours or failure modes that consumed significant search effort.

In practice, these summaries are valuable both diagnostically and scientifically. They help us quickly identify mistakes, dead ends, and other red flags, guiding redesign of the scoring rules, search space, or constraints for subsequent runs. They also provide a structured, human-readable report of how discoveries emerged. We include some of these AI-generated summaries for our runs in the appendices.

\subsection*{Declaration of use of AI}
This is an experimental mathematics paper, and a variety of AI-based tools were used extensively throughout the project. The AlphaEvolve user interface relies on gemini-2.5-pro and gemini-3-pro models to generate summaries of the problem setup, describe candidate strategies, suggest and write the evolve blocks and evaluation harnesses and also write AI summaries of the experiments. These capabilities were used extensively during the exploratory and iterative phases of the project. We also used code generation outside the AlphaEvolve system for vibe-coding evaluation harnesses. AlphaEvolve automatically produces detailed, Gemini-generated experiment reports, including performance summaries and qualitative analyses of evolved programs. These reports served as a starting point for several sections of this paper; however, the mathematical conclusions, conjectures, and counterexamples presented here arise from independent mathematical analysis. Except as otherwise indicated, the contents of this paper are the work of the author.

\subsection*{Acknowledgements}
I would like to thank the Google DeepMind AlphaEvolve team for the opportunity to be part of the first Trusted Users group with early access to AlphaEvolve, and for the continued support of the AlphaEvolve team throughout this project. 
I am especially grateful to Adam Zsolt Wagner, who was my first point of contact and an exceptionally helpful guide to the UI platform. I learned greatly from the many discussions and exchanges within the community of AlphaEvolve testers and mathematicians.

\section{AlphaEvolve as an experimental engine for mathematicians}
\label{sec:alphaevolve}

The experiments in this paper were carried out by using the evolutionary search model AlphaEvolve \cite{DeepMindAlphaEvolve2025,AlphaEvolveArxiv}. In this model the object being optimized through an evolutionary process  is not a numerical parameter but a (short) Python program.
The guiding idea is that many combinatorial existence problems admit a checkable formulation: once a candidate certificate is proposed (a reconstruction, an involution, a sequence of local moves), a fixed piece of code can decide whether it is correct and can usually quantify how far it is from being correct through a numerical scoring system.
AlphaEvolve uses the scoring system to guide the search engine by repeatedly rewriting the code that proposes candidates.

Fix a space $\mathcal{P}$ of short programs, for example Python functions of a prescribed signature.
In each experiment we distinguish two components.
\begin{enumerate}
    \item The first component is an \emph{evaluation harness}, a fixed module $H$ that encodes the mathematics of the problem.
Given a program $p\in\mathcal{P}$, the harness executes $p$ on a prescribed collection of instances and returns a structured report $H(p)=\bigl(s(p),\mathrm{aux}(p)\bigr)$,
where $s(p)\in\mathbb{R}$ is a scalar score used for selection and $\mathrm{aux}(p)$ is auxiliary data (for example counters, plots, or per-instance outcomes).
\item The second component is an \emph{evolvable block}, the part of the code that is allowed to change between iterations; in our experiments this was typically a single function (often called \texttt{find\_best\_certificate}) whose job is to propose a candidate object to be checked by $H$.
\end{enumerate}

AlphaEvolve repeatedly proposes small modifications $p\mapsto p'$ to the evolvable block and retains those modifications when the score improves.
One can view this as a form of program search guided by quantitative feedback: a reinforcement learning model on the state of python codes, where actions are LLM generated and the value of states is determined by the primary score. 
A detailed description of the AlphaEvolve model is given in \cite{AlphaEvolveArxiv,DeepMindAlphaEvolve2025,GeorgievGomezSerranoTaoWagner2025}.

\subsection{Harnesses as mathematical specifications}
For the purposes of mathematical work, it is helpful to think of the harness as \emph{immutable mathematics}.
Changing the evolvable block changes the search strategy; changing the harness changes the meaning of the experiment.
In particular, the harness specifies:
(i) the instance distribution or fixed test set;
(ii) the definition of validity (what outputs are rejected as meaningless);
(iii) the scoring system $s(p)$ which guides the search through the evolutionary process. .

In all three problems, the underlying mathematical statement is invariant under a large symmetry group (relabeling vertices, isotopys of Latin squares, change of basis in a vector space).
A naive score that evaluates only one representative per orbit encourages overfitting to arbitrary labels.
For that reason, the harnesses in this paper typically evaluate a candidate on several random symmetry actions and then aggregate by a \emph{worst-case} or near-worst-case statistic.
Concretely, if $\sigma_1,\dots,\sigma_S$ are independent random symmetry operations applied to the same underlying instance, we use scores of the form
\[
\mathrm{score}(\text{instance}) = \min_{1\le t\le S}  \mathrm{score}(\text{scramble}(t)),
\]
or a closely related robust aggregate.
This turns symmetry-invariance from a background fact into an explicit pressure in the objective.

A second recurring feature is the use of \emph{hard constraints} together with \emph{soft progress measures}.
For example, in the Alon-Tarsi experiments we treat "being a well-defined involution on the test set" as non-negotiable; among such candidates we then optimize the sign-flip rate and secondary diagnostics.
In the Rota experiments, success is rare enough that we add intermediate rewards for partial progress (such as increasing the number of valid columns) so that the score landscape is not completely flat.
These design choices are not cosmetic: without a shaped smooth scoring system, iterative program search tends to either stagnate or exploit loopholes in the evaluator.

Although the three mathematical problems look different, in each case the experimental search can be phrased as learning a sequence of small correcting steps.
In graph reconstruction, one tries to align unlabeled cards and incrementally fix a hypothesized adjacency matrix.
In the Latin-square setting, one applies local trades that change the parity while preserving the Latin property.
In the matroid setting, one performs local exchanges guided by short dependence certificates (circuits) in order to repair invalid columns.
At the level of the evolvable code, AlphaEvolve likewise proceeds by a sequence of local edits to the proposal program.
This local correction viewpoint was a practical guide in designing harnesses: we aimed to score not only final success, but also whether a candidate move makes a configuration more consistent with the constraints, so that incremental improvement is visible to the search.

\subsection{Toy AlphaEvolve experiment: packing equal circles in the unit square}
To make the above discussion concrete, consider the classical packing problem: place $N$ equal-radius circles in the unit square $[0,1]^2$ so as to maximize the common radius.
This is an optimization problem rather than a conjecture, but it has the same structure.

A natural harness represents a candidate by a list of centers $c_1,\dots,c_{N}\in[0,1]^2$ and defines the induced feasible radius
\[
r(c_1,\dots,c_{N}) = \min\Bigl( \min_{i}\{x_i,1-x_i,y_i,1-y_i\},\frac12\min_{i<j}\|c_i-c_j\|_2 \Bigr),
\]
where $c_i=(x_i,y_i)$.
The harness rejects malformed outputs, computes $r$, and returns $s=r$ (or $s=r$ minus a small penalty for numerical instability).
The evolvable block is then simply a function that proposes centers; AlphaEvolve iteratively rewrites that function so as to increase the returned radius.

The point of the example is that nothing in the loop depends on geometry in any special way.
Once the checker is correct, one is free to let the proposal code range from hard-coded patterns to randomized local improvement heuristics, and the system will empirically discover which modifications increase $r$ on the test protocol.
In the three problems of this paper, the certificate objects are more discrete, but the experimental logic is the same.

\subsection{What AlphaEvolve does and does not do}
It is important to emphasize what AlphaEvolve is \emph{not} doing in these experiments.
It is not proving theorems, and it is not replacing the mathematical step of turning a finite test-set phenomenon into a general statement.
What it provides is a way to explore a large space of \emph{explicit, checkable constructions} while keeping the meaning of success fixed in a harness.
The role of the mathematician remains to (i) design harnesses whose scores genuinely reflect the intended mathematics, (ii) interpret the resulting constructions as conjectures or proof sketches, and (iii) supply the human arguments needed to pass from experiments to theorems.

\section{The Graph Reconstruction Conjecture}
\label{sec:partIintro}

The reconstruction problem asks when a finite simple graph $G$ is determined up to isomorphism by the multiset
\[
\mathcal{D}(G)=\{G-v : v\in V(G)\}
\]
of vertex-deleted subgraphs (the \emph{deck} of $G$).
This is the setting of the classical Reconstruction Conjecture of Ulam and Kelly \cite{Ulam1960,Kelly1957}; see \cite{BondyHemminger1977} for survey-level background.

In this section we we report on using AlphaEvolve to search for explicit reconstruction procedures in two complementary regimes.
The first is a bipartite matrix deck model, in which a bipartite graph $G=(U\sqcup V,E)$ is represented by its $0$-$1$ incidence matrix $A\in\{0,1\}^{u\times v}$ and the deck consists of all row- and column-deletions of $A$.
The second is a planar regime, restricting to connected, biconnected planar graphs of minimum degree at least three that are not maximal planar.
In both settings the goal is algorithmic: given an unlabeled deck, output a graph $\widehat{G}$ (or matrix $\widehat{A}$) satisfying $\widehat{G}\cong G$.

\subsection{Experimental protocol and scoring}

A central experimental issue is to prevent label leakage.
If the cards of the deck inherit a consistent vertex ordering, reconstruction can exploit this hidden structure and does not address the genuine unlabeled problem.
To avoid this, each card is independently relabeled before being passed to the candidate reconstruction procedure.
In the bipartite setting we apply independent random permutations of rows and columns to each card, and in the planar setting we apply an independent random permutation of vertices.
Any successful method must therefore be permutation-invariant, up to the intrinsic automorphisms of the underlying graph.

Robustness is enforced by symmetrization.
For each underlying instance the harness generates $S$ independent relabelings of the same deck and aggregates by the worst case.
Thus if $\mathcal{D}^{(1)},\dots,\mathcal{D}^{(S)}$ are independent scramblings of the same deck and $H$ is the candidate reconstruction, the instance score is
\[
S_{\mathrm{instance}}(H)
=
\min_{1\le s\le S} S(H;\mathcal{D}^{(s)}).
\]
This strongly penalizes procedures that accidentally depend on a particular labeling.
The harness additionally stores a buffer of hard failures and replays them in later generations, forcing the search to deal with persistent hard cases instead of overfitting to the random generator.

We now describe the scoring functional $S(H;\mathcal{D})$ for a fixed scrambled deck $\mathcal{D}$.
The score takes values in $[0,1]$ and combines three normalized components.

The primary term is the \emph{deck-agreement score}
\[
s_{\mathrm{deck}}(H;\mathcal{D})
=
\frac{1}{n}
\sum_{[F]}
\min \left(
\mathrm{mult}_{\mathcal{D}}([F]),
\mathrm{mult}_{\mathcal{D}(H)}([F])
\right),
\]
where the sum ranges over isomorphism classes $[F]$ of $(n-1)$-vertex graphs and $\mathrm{mult}$ denotes multiplicity in the corresponding multiset.
Cards are compared up to graph isomorphism; in early experiments this was approximated by fast canonicalization procedures, and later replaced by exact isomorphism checks (for bipartite graphs via VF2-type backtracking with bipartition encoded as vertex colors).
The equality $s_{\mathrm{deck}}=1$ holds if and only if $\mathcal{D}(H)=\mathcal{D}$ as multisets.

Two auxiliary terms provide partial credit when exact reconstruction has not yet been achieved.
The \emph{degree score} $s_{\deg}$ compares the reconstructed degree multiset with that implied by the deck, and the \emph{edge-count score} $s_E$ compares total edge counts.
Both are normalized to lie in $[0,1]$ by linear absolute-error scaling.
These invariants are reconstructible from the deck and correlate with correctness, but are not sufficient for reconstruction.
Empirically they were essential in early generations, when exact deck agreement was too sparse a signal to guide the search.

The combined score on a single scramble is
\[
S(H;\mathcal{D})
=
w_{\mathrm{deck}} s_{\mathrm{deck}}
+
w_{\deg} s_{\deg}
+
w_E s_E,
\]
with $w_{\mathrm{deck}}$ chosen dominant.
Thus improvement of the deck-agreement term always outweighs any gain in auxiliary terms, while the latter smooth the search landscape before the first perfect reconstructions appear.

\subsubsection{Results}
After $\approx 2$ days of running, AlphaEvolve arrived to a reconstruction algorithm that achieved perfect reconstruction on the test benchmark of size $1420$:
\[
\text{successful reconstructions} = 1420/1420,\qquad
\text{worst-of-5-scrambles deck score} = 1.00\text{ for all cases.}
\]
No hard-negative failures were retained in the final generation.

For a mathematical reformulation of the evolved bipartite algorithm the final evolved bipartite method is best understood as a pipeline that begins by reconstructing \emph{global degrees} and \emph{neighbor-degree profiles} from the deck, then aggregates vertices into \emph{types} $(\deg,\text{profile})$ and derives the corresponding \emph{type-count constraints} describing how many neighbors of each opposite type each vertex must have. It then uses \emph{two-deletion compatibility signatures} to score candidate adjacencies and solves the resulting constrained assignment problem via min-cost flow separately on each type block. We now attempt to formalize the key invariants it exploits.

\begin{enumerate}
    \item \emph{Edge and degree reconstruction}
Let $n=u+v$, and write $m=|E(G)|$.
In the matrix model, $m$ equals the number of ones in $A$.
Let the deck consist of matrices $C_1,\dots,C_{u+v}$ with ones-counts $|C_t|_1$.
Each edge survives deletion of any vertex except its two endpoints, hence appears in exactly $n-2$ cards. Therefore
\[
m = \frac{1}{n-2}\sum_{t=1}^{u+v} |C_t|_1.
\]
For a row-deletion card corresponding to $u_i\in U$, we have
\[
\deg(u_i) = m - |C(u_i)|_1,
\]
and similarly for $v_j\in V$.

\item \emph{Neighbor-degree profile reconstruction (a key lemma)}
For a fixed $u\in U$, define the multiset
\[
P(u) = \{\deg(v) : v\in N(u)\}\quad\text{(with multiplicity)}.
\]
The evolved algorithm reconstructs $P(u)$ by comparing the degree multiset of $V$ in $G$ with the degree multiset of $V$ in the card $G-u$.

\begin{lemma}[Recovering neighbor-degree profiles from one deletion]\label{lem:profile}
Let $G=(U\sqcup V,E)$ be bipartite. Fix $u\in U$ and let $G' = G-u$.
Let $g_k$ be the number of vertices $v\in V$ with $\deg_G(v)=k$, and let $c_k$ be the number of vertices $v\in V$ with $\deg_{G'}(v)=k$.
Let $a_k$ be the number of neighbors $v\in N(u)$ with $\deg_G(v)=k$.
Then for all $k\ge 0$,
\[
c_k = g_k - a_k + a_{k+1}.
\]
Consequently, knowing the sequences $(g_k)_{k\ge 0}$ and $(c_k)_{k\ge 0}$ determines $(a_k)_{k\ge 0}$ uniquely by downward recursion.
\end{lemma}

\begin{proof}
Deleting $u$ decreases the degree of a vertex $v\in V$ by $1$ iff $v\in N(u)$, and leaves it unchanged otherwise.
Hence, among vertices in $V$ of degree $k$ in $G$, exactly $a_k$ drop to degree $k-1$ in $G'$. This contributes $-a_k$ to the count at degree $k$ in $G'$.
Meanwhile, among degree-$(k+1)$ vertices in $G$, exactly $a_{k+1}$ drop into degree $k$ in $G'$, contributing $+a_{k+1}$.
All other vertices retain their degree. Thus $c_k = g_k - a_k + a_{k+1}$.
If degrees are bounded by $\Delta$, set $a_{\Delta+1}=0$ and solve backwards:
$a_\Delta = g_\Delta - c_\Delta$, then $a_{\Delta-1}= g_{\Delta-1}+a_\Delta - c_{\Delta-1}$, etc.
\end{proof}

The algorithm defines a \emph{type} for each vertex as
\[
\tau(u) = (\deg(u), P(u)).
\]
This typically refines degrees substantially in generic random regimes.

\item \emph{Type-count constraints and blockwise flow}
Let $\mathcal{T}_U$ be the multiset of types of $U$ vertices and $\mathcal{T}_V$ the multiset of types of $V$ vertices.
From a row-deletion card $G-u$, one can count how many $V$-vertices of each type appear \emph{in that card}.
By Lemma~\ref{lem:profile}, the algorithm can infer how many $V$-vertices of each global type must have been adjacent to $u$ (because adjacency shifts the neighbor's degree and neighbor-profile in the card in a predictable way).

This yields, for each $u\in U$ and each type $t\in \mathcal{T}_V$, a requirement number
\[
R(u,t) = \#\{ v\in V : \tau(v)=t,\ (u,v)\in E\}.
\]
Similarly one obtains $R(v,s)$ for $v\in V$ and $s\in\mathcal{T}_U$.

Given these constraints, reconstruction reduces to: for each pair of types $(s,t)$, determine a bipartite subgraph between $U_s=\{u:\tau(u)=s\}$ and $V_t=\{v:\tau(v)=t\}$ such that each $u\in U_s$ has exactly $R(u,t)$ neighbors in $V_t$ and each $v\in V_t$ has exactly $R(v,s)$ neighbors in $U_s$.
This is a feasible transportation/degree-constrained bipartite realization problem, solvable by flow when a feasible solution exists.

\item \emph{Two-deletion compatibility scoring}
To select among multiple feasible realizations, the evolved algorithm uses a "compatibility" score for each candidate edge $(u,v)$ computed from two-deletion subcards:
\[
G-u-v
\]
viewed from both the $u$-deleted and $v$-deleted sides.
Operationally it builds multisets of signatures of $(u,v)$-double-deleted cards grouped by local types, and compares intersections (a Jaccard/F1-like overlap) under the hypotheses "edge" vs.\ "non-edge." The resulting confidence scores become costs in the min-cost flow.
Mathematically, this is an \emph{empirical estimator} for which adjacency hypothesis makes the induced multiset of two-deletion cards most consistent across the deck.
\end{enumerate}
We state a conjecture intentionally aligned with what the harness actually tested (2-connected, non-regular, moderately dense bipartite graphs).

\begin{conjecture}[Generic bipartite reconstruction by low-order type constraints]\label{conj:bipartite}
Let $G$ be a bipartite graph with parts $(U,V)$, $|U|=u$, $|V|=v$.
Assume $G$ is 2-connected, not regular, and has minimum degree at least $6$.
Then, for generic such $G$ (e.g.\ under moderate-density random models conditioned on these constraints), the deck $\mathcal{D}(G)$ determines $G$ uniquely up to bipartite isomorphism, and $G$ can be reconstructed in polynomial time from $\mathcal{D}(G)$ by recovering degrees and neighbor-degree profiles from the deck, forming types $\tau=(\deg,P)$ and type-count constraints $R(\cdot,\cdot)$, and selecting the unique feasible adjacency by maximizing consistency of two-deletion cards (equivalently, minimizing a flow cost derived from two-deletion compatibility).
\end{conjecture}

\begin{remark}
This conjecture is \emph{not} the full reconstruction conjecture for bipartite graphs. It is an experimentally suggested \emph{mechanism}: that in high-degree, 2-connected, non-regular bipartite graphs, low-order invariants already pin down the graph with high probability, and the remaining ambiguity is resolved by two-deletion consistency.
\end{remark}

\subsubsection{A conditional correctness statement for the bipartite algorithm}
The evolved bipartite method will be correct under sufficiently strong uniqueness conditions, namely that types $\tau=(\deg,P)$ separate vertices well enough that the type-count constraints $R(\cdot,\cdot)$ yield a small feasible set of realizations, and that two-deletion signatures separate the remaining ambiguity by a strict objective gap.
One can formalize this as: if the induced transportation constraints admit a unique feasible adjacency (up to permutations within identical types), then the algorithm reconstructs it. Proving such uniqueness for wide graph families is nontrivial but may be approachable for random bipartite graphs (via concentration and collision bounds for degree/profile signatures).

\subsection{Planar graph reconstruction}

The planar experiments targeted graphs $G$ that are planar, connected and 2-connected (biconnected), have minimum degree $\delta(G)\ge 3$, and are \emph{not} maximal planar (so $m < 3n-6$ for $n\ge 3$).
For background on reconstruction and recognizability questions in planar regimes, see \cite{BilinskiKwonYu2007,BondyHemminger1977}.

\subsubsection{The evaluation harness}
Early planar harnesses compared cards by Weisfeiler-Lehman hashing, which is isomorphism-invariant but not complete (hash collisions possible).
This creates the risk that an evolved algorithm merely learns the evaluator's equivalence relation.
Accordingly, the harness was hardened by scrambling each card by an independent random vertex permutation, by comparing reconstructed decks to target decks via \emph{exact} isomorphism or canonical-labeling equality, and by evaluating each underlying graph under multiple independent scrambles and taking the worst-case score.

The final planar harness constructed a bucketed test set (roughly 40\% generic and 60\% targeted hard buckets), plus a persistent buffer of any failed instances for hard-negative mining.
Graphs were generated by first building an Apollonian triangulation (a maximal planar graph) via iterative face subdivision and then deleting edges while maintaining planarity, biconnectedness, and $\delta\ge 3$, with an explicit enforcement of non-maximality. The harness sampled across sizes $n\in\{12,14,16,18,20,22,24,26\}$ and evaluated each instance under $S=5$ independent scrambles, taking worst-case performance.

\subsubsection{The scoring system}
The planar reconstruction harness follows the same mathematical template as in \S\ref{sec:partIintro}, but the need for speed forced an explicit staging of the evaluator.

For a target planar graph $G$ on $n$ vertices and a proposed reconstruction $\widehat{G}$, the fundamental quantity is the deck-agreement score
\[
s_{\mathrm{deck}}(\widehat{G})\in[0,1],
\]
which measures how well the multiset $\mathcal{D}(\widehat{G})$ matches the target deck $\mathcal{D}(G)$.
In early experiments, the harness compared cards using fast isomorphism-invariant surrogates (such as Weisfeiler-Lehman hashes), which makes $s_{\mathrm{deck}}$ cheap but not logically complete.
In later experiments, once evaluator-exploitation became a realistic failure mode, the harness was hardened by switching to exact deck equality up to isomorphism, at least on the smaller buckets of the test set.
Throughout, robustness to labeling was enforced by evaluating several independent scrambles of the same underlying instance and taking the worst-case score across scrambles.

As in the bipartite setting, we supplemented $s_{\mathrm{deck}}$ with cheaper global diagnostics (degree data, edge count, and small subgraph counts) to avoid a completely flat score landscape before any perfect reconstructions appear.
These auxiliary terms were kept subordinate to the deck term so that improving the score could not systematically move a candidate away from true reconstruction.
Finally, the bucketed test-set design (see below) and hard-negative mining ensured that improvements reflected genuine generalization across structurally distinct planar graphs rather than specialization to a single easy regime.

\subsubsection{Results}
A striking outcome was that AlphaEvolve quickly found \emph{very simple} exact-verification algorithms that achieved perfect reconstruction on the hardened test set.
A representative evolved method proceeds by using Kelly-style counting to reconstruct $m=|E(G)|$ and then computing degrees from the cards via $\deg(v)=m-|E(G-v)|$. Let $\delta=\min_v \deg(v)$ and choose a card $G-v^\star$ with $\deg(v^\star)=\delta$. It then enumerates candidate neighbor sets $S\subseteq V(G-v^\star)$ of size $\delta$ such that adding $v^\star$ adjacent to $S$ yields the correct global degree multiset. Add an additional filter using triangle counts, reconstructible via
\[
T(G)=\frac{1}{n-3}\sum_{v} T(G-v),
\]
and noting that new triangles created by adding $v^\star$ equal the number of edges induced by $S$. For each candidate, build $\widehat{G}$ and accept it if $\mathcal{D}(\widehat{G})=\mathcal{D}(G)$ under exact deck equality.

The key structural reason this is fast for planar graphs is Euler's bound: every planar graph has a vertex of degree at most $5$, so $\delta\le 5$ in the connected planar regime. Hence the neighborhood enumeration step is $O(n^5)$ in the worst case and is typically far smaller in practice.

With the hardened planar harness, the user reports perfect reconstruction in about an hour of evolution:
\[
\text{successful reconstructions} = 896/896,\quad
\text{worst-of-5-scrambles deck score}=1.00\text{ globally}.
\]
In a stronger, "tricky planar set" variant, AlphaEvolve again found a perfect exact-verification algorithm in minutes.
A conjecture suggested by the planar experiments is the following. 

\begin{conjecture}[Biconnected planar $\delta\ge 3$ non-maximal reconstruction by bounded-degree search]\label{conj:planar}
Let $G$ be a planar biconnected graph with $\delta(G)\ge 3$ that is not maximal planar.
Then $G$ is reconstructible from its deck.
Moreover, there exists a reconstruction algorithm that recovers $m$ and the degree multiset from the deck, selects a minimum-degree vertex $v$ (with $\deg(v)\le 5$ by planarity), identifies a small set of candidate neighborhoods for $v$ in some card $G-v$ using degree and triangle constraints, and verifies the correct candidate by exact deck equality, running in time polynomial in $n$ (with dependence dominated by $O(n^5)$ neighborhood search and deck isomorphism checks).
\end{conjecture}

\begin{remark}
The literature on reconstruction in planar regimes is extensive, and the experiments here should be read as computational probes rather than as substitutes for structural theory.
For discussion of reconstructibility and recognizability questions within planar graph classes, see \cite{BilinskiKwonYu2007} and the general survey \cite{BondyHemminger1977}.
\end{remark}

\subsubsection{A bounded-degree reconstruction route for planar graphs}
For planar graphs, Conjecture~\ref{conj:planar} is plausible because there always exists a vertex $v$ with $\deg(v)\le 5$, the degree multiset constraints drastically restrict candidate neighborhoods of $v$ in $G-v$, triangle counts and higher small-subgraph counts add further restrictions, and exact deck equality provides a definitive verification step.

A proof strategy would require showing that for every $G$ in the target class, there exists at least one vertex $v$ such that its neighborhood in some card is uniquely identifiable by reconstructible invariants of bounded order, or at least that the number of consistent candidates is always bounded by a polynomial so that exact verification yields a polynomial-time reconstruction procedure.

\section{The Alon-Tarsi conjecture}
\label{sec:partIIintro}

The Alon-Tarsi conjecture asserts that for even order $n$ the signed enumeration of Latin squares is nonzero, equivalently that $E_n-O_n\neq 0$ where $E_n$ and $O_n$ denote the numbers of even and odd Latin squares under the standard Alon-Tarsi sign convention.
A classical approach to such statements is to construct a sign-reversing involution on most Latin squares, leaving a residual set whose signed sum can be analyzed.
Our experiments follow this philosophy: the evolvable code block specifies a deterministic map $F:\mathcal{L}_n\to\mathcal{L}_n$, and the harness scores $F$ by how well it behaves like a canonical sign-reversing involution.

Across all experiments, validity of Latin squares and involutions are treated as hard constraints, checked by the harness.
Given a test set $\mathcal{T}_n\subseteq\mathcal{L}_n$, the harness checks that $F(L)\in\mathcal{L}_n$ for every $L\in\mathcal{T}_n$ and that $F(F(L))=L$; candidates failing these conditions receive negligible fitness.
Among valid involutions, the primary objective is the sign-flip rate
\[
p_{\mathrm{flip}}(F):=\frac{1}{|\mathcal{T}_n|}\sum_{L\in\mathcal{T}_n}\mathbf{1}\bigl[\sgn(F(L))=-\sgn(L)\bigr],
\]
together with the non-identity rate (to discourage the trivial fixed-point map).
To align the search with the involution philosophy, the harness also tracks the \emph{residual set}
\[
\Res(F):=\{L\in\mathcal{T}_n:\ \sgn(F(L))=\sgn(L)\}\ \cup\ \{L\in\mathcal{T}_n:\ F \text{ fails a hard constraint on }L\},
\]
and computes a residual bias statistic $b(F)=\frac{1}{|\Res(F)|}\sum_{L\in\Res(F)}\sgn(L)$ (reported as $b(F)^2$).
A residual bias close to $1$ indicates that the unpaired objects are overwhelmingly of a single sign, which is precisely the signature one would hope to prove in a residual-set approach.

Although it is conceptually cleaner to think of $(p_{\mathrm{flip}},q,b^2,\dots)$ as a multi-criterion report, AlphaEvolve ultimately needs a scalar value to decide whether a code modification is better.
In practice the harness uses a tiered (effectively lexicographic) scalarization: hard-constraint violations are assigned a very low score, so that any candidate that is not a well-defined involution on the test set is immediately dominated.
Among valid involutions, the dominant contribution comes from $p_{\mathrm{flip}}$, with secondary terms (nontriviality, residual bias, locality, and symmetry robustness, depending on the experiment) used only for tie-breaking and for gently shaping the landscape.
This design choice is important for mathematical interpretability: improving the score can be read as improving a specific quantitative feature of an otherwise well-formed combinatorial map.

Experiment~2 adds an explicit locality term in order to prefer \emph{proof-shaped} involutions built from small trades.
Writing $\mathrm{supp}(F;L)$ for the number of cells modified by $F$ on $L$, the harness defines a quality score
\[
q(F):=\frac{1}{|\mathcal{T}_n^{\mathrm{succ}}|}\sum_{L\in\mathcal{T}_n^{\mathrm{succ}}}\exp \Bigl(-\frac{\mathrm{supp}(F;L)}{2n}\Bigr),
\]
where $\mathcal{T}_n^{\mathrm{succ}}$ denotes the subset on which $F$ is valid, involutive, and sign-flipping.

Experiment 3 adds a decisive anti-overfitting stress test: before applying $F$, each input $L$ is conjugated by a deterministic pseudorandom isotopy (permutations of rows, columns, and symbols).
For even $n$ the Alon-Tarsi sign is invariant under isotopy, so this preserves the true value of the sign-flip condition while destroying any reliance on fixed coordinates.
In effect, the scoring protocol forces AlphaEvolve to discover selection rules that are canonical under the full isotopy group.

\subsection{Experiment 1: evolving a nontrivial sign-flipping involution with biased residual set}
Experiment~1 evaluated candidate maps $F$ on precomputed test sets for orders $n\in\{8,10,12,14,26\}$ (with runtime sanitization to reject malformed outputs). AlphaEvolve found a near-perfect sign-flipping involution: the residual bias squared equaled $1$ in all tested orders, meaning the residual set (though small) was \emph{completely one-sided in sign} on the test data.

\subsubsection{Results and analysis}
The evolved code constructs many candidate \emph{two-line trades} (on row pairs and column pairs) and applies the first trade that is \emph{stable} under a canonical-order check.

The core primitive is a two-line cycle trade. It proceeds by choosing a pair of rows (or columns), forming the induced permutation on indices obtained by matching equal symbols between the two lines, decomposing this permutation into cycles, and then swapping the two lines exactly on the positions belonging to a selected cycle (or union of cycles). An odd cycle yields a sign-flipping trade, whereas an even cycle yields a sign-preserving trade.

The algorithm also introduced a rare sign-preserving "sabotage" option (even-cycle trades accepted by a deterministic hash filter), and a sign-preserving fallback (symbol swap $0\leftrightarrow 1$ for certain $n$, or a row-symbol conjugation for others).

The reported metrics were:

\begin{table}[h]
\centering
\begin{tabular}{@{}lccccc@{}}
\toprule
 & $n=8$ & $n=10$ & $n=12$ & $n=14$ & $n=26$ \\
\midrule
Validity rate & 1.00 & 1.00 & 1.00 & 1.00 & 1.00 \\
Involution rate & 1.00 & 1.00 & 1.00 & 1.00 & 1.00 \\
Sign-flip rate & 0.98 & 0.90 & 0.89 & 0.89 & 0.75 \\
Residual bias$^2$ & 1.00 & 1.00 & 1.00 & 1.00 & 1.00 \\
\bottomrule
\end{tabular}
\caption{Experiment 1: summary of the reported per-$n$ metrics.}
\end{table}

Additionally, the harness printed residual samples showing $\lvert\Res\rvert$ extremely small (1 or 2) and \emph{entirely} even for $n=8,10,14,26$ and entirely odd for $n=12$ on the test sets.

\[
|\Res_8|=1,\quad |\Res_{10}|=1,\quad |\Res_{12}|=2,\quad |\Res_{14}|=1,\quad |\Res_{26}|=1,
\]
and moreover all residual squares had the same sign within each order:
$\sgn(L)=+1$ for $L\in\Res_8\cup\Res_{10}\cup\Res_{14}\cup\Res_{26}$, while $\sgn(L)=-1$ for all $L\in\Res_{12}$.

On the tested data, the evolved map behaves like a sign-reversing involution on almost all inputs; the failures form a tiny residual set, but with maximal sign bias.
This is exactly the involution philosophy: a perfect sign-reversing involution would force $E_n-O_n=0$, while a \emph{partial} sign-reversing involution leaves a residual sum equal to $E_n-O_n$ restricted to the residual.
Experiment~1 suggests that a carefully chosen \emph{canonical cycle trade} can pair most squares, while the leftover squares may be forced into a single sign class.

\subsubsection{Mathematical reformulation of the evolved algorithm}

The evolved map searches for a two-line Latin trade supported on cycles of a naturally associated permutation, and applies the first ``stable'' trade found in a fixed scan order.  We describe the row version; the column version is obtained by transposing the roles of rows and columns.

\begin{definition}[Row-matching permutation]
Let $L$ be a Latin square of order $n$, and fix two distinct rows $r_1\neq r_2$.
Define a permutation $\pi=\pi_{r_1,r_2}\in S_n$ of the column set $\{0,1,\dots,n-1\}$ by the rule:
for each column $c$, the value $\pi(c)$ is the unique column $c'$ such that
\[
L_{r_2,c'}=L_{r_1,c}.
\]
Equivalently, $\pi$ sends $c$ to the location in row $r_2$ of the symbol appearing in row $r_1$ at column $c$.
\end{definition}

\begin{definition}[Cycle trade on a row pair]
Let $C\subseteq\{0,1,\dots,n-1\}$ be a union of cycles of $\pi_{r_1,r_2}$.
Define $T_{r_1,r_2,C}(L)$ to be the array obtained by swapping the entries of rows $r_1$ and $r_2$ in exactly the columns in $C$, leaving all other entries unchanged.
\end{definition}

\begin{lemma}[Validity and involution]
If $C$ is a union of cycles of $\pi_{r_1,r_2}$, then $T_{r_1,r_2,C}(L)$ is again a Latin square.
Moreover $T_{r_1,r_2,C}$ is an involution:
\[
T_{r_1,r_2,C}\bigl(T_{r_1,r_2,C}(L)\bigr)=L.
\]
\end{lemma}

\begin{lemma}[Effect on the Alon-Tarsi sign for a single cycle]
Assume $C$ is a single cycle of $\pi_{r_1,r_2}$ of length $\ell$.
Then
\[
\sgn\bigl(T_{r_1,r_2,C}(L)\bigr)=(-1)^{\ell}\,\sgn(L).
\]
In particular, an odd cycle produces a sign flip, while an even cycle preserves the sign.
\end{lemma}

\begin{proof}
In each affected column $c\in C$, the trade swaps the two entries in that column lying in rows $r_1$ and $r_2$; this is a transposition in the column permutation and hence flips the sign of that column.  Thus the product of the column signs acquires the factor $(-1)^{|C|}=(-1)^{\ell}$.

On the other hand, restricting to a single $\ell$-cycle of $\pi_{r_1,r_2}$, the operation replaces the $r_1$-row entries on $C$ by the corresponding $r_2$-row entries, which permutes the symbols on those $\ell$ positions by an $\ell$-cycle; this changes the sign of row $r_1$ by $(-1)^{\ell-1}$, and similarly changes the sign of row $r_2$ by $(-1)^{\ell-1}$.  The product of row signs therefore acquires the factor $(-1)^{2(\ell-1)}=+1$.  Multiplying row and column contributions gives the stated formula.
\end{proof}

The evolved map fixes a finite deterministic sequence $\mathcal{S}$ of candidate two-line trades.
Each element of $\mathcal{S}$ specifies (i) whether to work in row mode or column mode, (ii) a pair of line indices, and (iii) a selection rule for the support $C$:
either the \emph{flip} rule (choose the lexicographically first odd cycle of $\pi$, and optionally adjoin the first even cycle), or the \emph{sabotage} rule (choose the first even cycle, but accept it only under a deterministic rarity filter, producing sign-preserving moves on a small subfamily).
In the implementation, $\mathcal{S}$ is generated using a pseudorandom number generator with fixed seed, hence is deterministic.

Given $L$, the algorithm scans $\mathcal{S}$ in order and, whenever the specified rule produces a valid cycle trade, forms the tentative output $L'=T(L)$.
It accepts $L'$ only if it is \emph{stable} with respect to the scan order, meaning that \emph{no earlier} candidate trade from $\mathcal{S}$ becomes available when tested on the tentative output $L'$.
If a stable trade is found, it is returned.

If no stable trade is found, the algorithm applies a fixed fallback involution depending on $n$:
when $n/2$ is odd it swaps the symbols $0\leftrightarrow 1$ (a sign-preserving involution for even $n>1$), while when $n/2$ is even it applies row-symbol conjugation (the $S_3$-conjugate obtained by exchanging the column and symbol coordinates).

\begin{conjecture}[Residual-bias involution principle suggested by Experiment~1]
For each even $n$, there exists a deterministic scan order $\mathcal{S}_n$ of two-line cycle trades together with a deterministic stability rule of the above no earlier trade becomes available type, such that the induced map $F_n:\mathcal{L}_n\to\mathcal{L}_n$ is an involution and satisfies
\[
\sgn(F_n(L))=-\sgn(L)\quad\text{for all }L\in\mathcal{L}_n\setminus\Res_n,
\]
where the residual set $\Res_n$ is nonempty and has nonzero signed sum, for instance $\sum_{L\in\Res_n}\sgn(L)=\pm|\Res_n|$.
In particular, $\sum_{L\in\mathcal{L}_n}\sgn(L)\neq 0$, which would imply the Alon-Tarsi conjecture at order $n$.
\end{conjecture}

\subsection{Experiment 2.}
Experiment~2 strengthened the evaluation protocol in two directions: it increased the size and adversarial nature of the test data (100 examples per fro each $n$ for $n\in\{8,10,12,14,26\}$), and it refined the score to prefer local, analyzable trades.
\subsubsection{Harness summary}
A trial on an input $L$ is counted as a \emph{success} only if the candidate map is valid on $L$, involutive on $L$ (including non-identity), and flips sign on $L$.
Beyond the binary success indicator, the harness also computes a \emph{quality} score that rewards locality.
Writing $\mathrm{support}(F;L)$ for the Hamming support size (the number of cells changed by the map on $L$), the per-instance quality is
\[
q(L)=\exp\!\Bigl(-\frac{\mathrm{support}(F;L)}{2n}\Bigr),
\]
and the reported quality is the average of $q(L)$ over successful instances.
The choice of an exponential decay makes the score sensitive to multiplicative improvements in support size and naturally keeps $q(L)\in(0,1]$, while the $2n$ scaling reflects the fact that the simplest two-line trades have supports proportional to $n$.
In effect, the fitness function promotes "proof-shaped" involutions that act by small canonical trades rather than by large global rearrangements.

\subsubsection{Results and analysis}
The evolved algorithm generalizes the cycle-trade idea in three ways. It searches not only row-pairs but also column-pairs and \emph{symbol-pairs} (a conjugate view in which one swaps symbols in selected columns by looking at where symbols occur). It tries several deterministic pairing families (adjacent pairs, half-shifted pairs, pairs from a 1-factorization, and exhaustive all-pairs). It chooses the \emph{smallest} odd cycle available among the pairs considered at a given stage, to bias toward local modifications.

Reported metrics:

\begin{table}[h]
\centering
\begin{tabular}{@{}lccccc@{}}
\toprule
 & $n=8$ & $n=10$ & $n=12$ & $n=14$ & $n=26$ \\
\midrule
Validity rate & 1.00 & 1.00 & 1.00 & 1.00 & 1.00 \\
Success rate  & 0.85 & 0.98 & 0.98 & 0.97 & 1.00 \\
Quality (avg) & 0.69 & 0.73 & 0.77 & 0.80 & 0.89 \\
\bottomrule
\end{tabular}
\caption{Experiment 2: reported success and locality (quality) by order.}
\end{table}

The near-perfect success for $n\ge 10$ and perfect success for $n=26$ on this dataset strongly suggests that
odd-cycle trades are ubiquitous once one searches across conjugate views and pairing families.
The dip at $n=8$ indicates either (i) genuine absence of an odd-cycle trade in the searched families for some inputs, or
(ii) a determinism/involution interaction where applying a trade creates an earlier trade and breaks $F(F(L))=L$ in a minority of cases.
The quality score increasing with $n$ is consistent with the existence of relatively small supports (on a $2n$ scale) among successful trades.

\subsubsection{Mathematical reformulation of the evolved algorithm}
Let $L$ be a Latin square of order $n$.  The algorithm operates in three conjugate line families:
\begin{itemize}
\item \emph{Row view.} Lines are rows $r\in[n]$; a trade swaps two rows on a selected set of columns.
\item \emph{Column view.} Lines are columns $c\in[n]$; a trade swaps two columns on a selected set of rows.
\item \emph{Symbol view.} Lines are symbols $s\in[n]$; in each column $c$, each symbol occurs in a unique row.  Swapping symbols $s_1,s_2$ in a chosen set of columns means swapping their positions within those columns.
\end{itemize}

In each view, for a chosen pair of lines $(\ell_1,\ell_2)$ one defines the induced permutation on the complementary index set (columns for the row and symbol views; rows for the column view), decomposes it into cycles, and observes that any \emph{odd} cycle yields a sign-flipping two-line trade. The algorithm found by AlphaEvolve can be describes as follows. 

Fix an order of conjugate types (rows, columns, symbols) and an order of pairing strategies (adjacent, half-shifted, 1-factorization, exhaustive).
For each conjugate type in that order, and for each pairing strategy in that order:
enumerate candidate pairs $(\ell_1,\ell_2)$ in the strategy's canonical order; for each pair, compute the induced permutation and list its odd cycles.
Among all odd cycles found, choose one minimizing cycle length, breaking ties deterministically by the pair's ordinal position and then by lexicographic cycle data.
Apply the corresponding cycle trade and return.
If no odd cycle is found in any view, apply a fixed fallback involution (in the evolved code: swap rows $0$ and $1$).

Experiment~2 suggests the following conjecture.
\begin{conjecture}
Fix even $n$.  Every Latin square $L\in\mathcal{L}_n$ admits an odd-cycle two-line trade in at least one of the three conjugate views (row, column, symbol).  Moreover, such a trade can always be found among pairs belonging to a fixed canonical 1-factorization of the $n$ lines in that view.
\end{conjecture}

A proof of this conjecture, together with a stability/canonicalization argument ensuring involutivity under repeated application, would provide a direct path to a sign-reversing involution on a large subset of $\mathcal{L}_n$.

\subsection{Experiment 3: a stabilized deterministic sign-flipping involution}
Experiment~3 introduced a decisive stress test: before the evolved map $F$ was applied, each input Latin square $L$ was transformed by a deterministic but pseudorandom \emph{isotopy}:
a permutation of rows, columns, and symbols. For even $n$, the Alon-Tarsi sign is invariant under isotopy, so this does not change the truth of $\sgn(F(L))=-\sgn(L)$ but prevents any method that relies on memorizing specific instances or rigid coordinate positions.

The reported metrics were the validity rate, the involution rate, the sign-flip rate, and the non-identity rate.

\subsubsection{Results and analysis}
The evolved algorithm is substantially more proof-shaped than the earlier ones. It combines a \emph{tiered search} for trades (starting with disjoint adjacent row pairs, then adjacent column pairs, then a global exhaustive search), with a robust \emph{canonical cycle representative} that is stabilized under reversal and rotation so that the same cycle is selected after applying the trade. It also enforces a \emph{stability predicate} rejecting trades whose application would create the selection rule changes on the second application failure mode, and it adds a further \emph{canonical isotopy layer}: if no stable trade exists in the current coordinates, the algorithm applies a deterministic pseudorandom isotopy and searches again, but only when no simpler tier would have worked on the resulting square. This is an explicit realization of the paradigm
\[
F(L)=P^{-1}\circ T\circ P(L),
\]
with $P$ chosen canonically and with stability checks engineered so that the same $P$ is chosen for $L$ and $F(L)$.

Reported metrics:

\begin{table}[h]
\centering
\begin{tabular}{@{}lccccc@{}}
\toprule
 & $n=8$ & $n=10$ & $n=12$ & $n=14$ & $n=26$ \\
\midrule
Validity rate & 1.00 & 1.00 & 1.00 & 1.00 & 1.00 \\
Involution rate & 1.00 & 1.00 & 1.00 & 1.00 & 1.00 \\
Non-identity rate & 1.00 & 1.00 & 1.00 & 1.00 & 1.00 \\
Flip rate & 0.98 & 1.00 & 1.00 & 1.00 & 1.00 \\
\bottomrule
\end{tabular}
\caption{Experiment 3: reported per-$n$ metrics under isotopy randomization.}
\end{table}

The evolved map is always valid, always an involution, always nontrivial, and flips sign essentially always on the dataset (perfect for $n\ge 10$ and $98\%$ for $n=8$).

This is a good evidence that the stabilized odd-cycle trade and canonical isotopy framework is close to a genuine combinatorial involution.

\subsubsection{Mathematical reformulation of the evolved algorithm}
We extract the mathematical skeleton.

\paragraph{Step A: stable trade selector $T(L)$.}
Define $T(L)$ to be the result of applying the first available odd-cycle trade in the priority list:
One first searches adjacent disjoint row pairs $(0,1),(2,3),\dots$ for the lexicographically smallest stabilized odd cycle. If none is available, one searches adjacent disjoint column pairs $(0,1),(2,3),\dots$ with the further restriction that the chosen trade does \emph{not} create a tier-1 row-pair trade after application. Otherwise, among \emph{all} row-pair and column-pair odd-cycle trades, one selects the minimal trade key (cycle length, stabilized cycle data, indices, and mode) but accepts it only if (i) the output is \emph{stable} (admits the same trade on reapplication) and (ii) no strictly smaller trade key would become stable after application.

If no such stable trade exists, declare $T(L)$ undefined.

\paragraph{Step B: canonical isotopy layer.}
Let $\mathcal{P}_1$ be the set of row transpositions $(r_1r_2)$ in lexicographic order,
and $\mathcal{P}_2$ the set of column transpositions $(c_1c_2)$ in lexicographic order.
Define $F(L)$ by:
We define $F(L)$ by a deterministic cascade. If $T(L)$ is defined, set $F(L)=T(L)$. Else, find the first $P\in\mathcal{P}_1$ such that $T(P(L))$ is defined; let $L':=P^{-1}(T(P(L)))$ and accept it only if $T(L')$ is undefined (no simpler tier-0 solution). If accepted, set $F(L)=L'$. Else, find the first $P\in\mathcal{P}_2$ such that $T(P(L))$ is defined; let $L':=P^{-1}(T(P(L)))$ and accept it only if neither $T(L')$ nor any row-transposition tier yields a solution.
If accepted, set $F(L)=L'$. Else, apply a fixed fallback involution (the evolved code uses a symbol transposition).

By construction, each accepted move is an involutive trade conjugated by an involutive isotopy, so it is involutive \emph{provided} the selection is stable under reapplication. The extensive stability checks are precisely what enforces this.

A conjecture suggested by Experiment 3 is the following.
\begin{conjecture}[Stable odd-cycle trade conjecture with bounded isotopy depth]
Fix even $n$. For every Latin square $L\in\Ln$, there exists a stable odd-cycle two-line trade in $L$ or in a conjugate $P(L)$ where $P$ is a single transposition of two rows or two columns, such that the stabilized canonical selection rule chooses the same $(P,\text{trade})$ for $L$ and for $F(L)$.
Consequently, the resulting map $F:\Ln\to\Ln$ is a sign-reversing involution on all of $\Ln$ except possibly a residual set $\Res_n$ handled by a final canonical rule; moreover $\sum_{L\in\Res_n}\sgn(L)\neq 0$.
\end{conjecture}

This conjecture is a concrete, algorithmically precise strengthening of the residual-bias involution idea.
If established, it would yield a constructive proof strategy for the Alon-Tarsi conjecture.

\section{Rota's Basis Conjecture}
\label{sec:partIIIintro}

Rota's Basis Conjecture (RBC) asks whether $n$ disjoint bases of a rank-$n$ matroid can always be rearranged into an $n\times n$ grid whose columns are also bases.
The experiments in this part restrict to linear matroids over $\mathbb{F}_2$ and aim to construct $n$ disjoint transversal bases by a sequence of local moves.
The evolvable block in these experiments is not a direct output construction but a \emph{policy} for a greedy local-exchange engine: at each step, given the current partially filled grid, the policy scores all legal insert and repair moves and the engine executes a move of maximal score.

The evaluation harness converts the qualitative goal "find a full transversal-basis grid" into a scalar fitness that both rewards success and shapes the intermediate dynamics.
For a rollout producing a trajectory $(A_t)_{t=0}^T$ of assignment matrices, the primary term is the success indicator $\mathbf{1}[\text{all columns are full bases at termination}]$.
To avoid sparse all-or-nothing feedback, the harness also tracks progress and damage: it rewards completing full columns, penalizes breaking a column that was already full, penalizes the use of explicit repair moves, and grants partial credit for the number of valid (independent) columns at the end of the rollout.
A small time penalty proportional to the number of steps discourages policies that solve instances only by wandering for long trajectories.
The concrete fitness scalar in the experiments below has the form displayed in Experiment~A, with fixed weights chosen so that success dominates all secondary terms.

Robustness is enforced by evaluating each instance under multiple randomized rollouts and aggregating by majority vote, and by sampling both generic and deliberately structured trap instances from circuit-rich pools.

\subsection{Common formalism.}

Fix $n\ge 2$.
We work in a linear matroid over $\mathbb{F}_2$ of rank $n$ with $n$ given bases.

\begin{definition}[State as an assignment matrix]
Let $B_1,\dots,B_n$ be bases, where
$B_i=\{v_{i,1},\dots,v_{i,n}\}$ (elements are distinguished by their row-index even if vectors coincide).
A \emph{state} is an $n\times n$ matrix $A=(A_{ij})$ with entries in $\{-1,1,\dots,n\}$ such that, for each row $i$, the non-$-1$ entries are all distinct.
The \emph{column set} is
\[
T_j(A)=\{v_{i,A_{ij}}: A_{ij}\ne -1\}.
\]
Column $j$ is \emph{valid} if $T_j(A)$ is independent, and \emph{full} if it is valid and $|T_j(A)|=n$.
\end{definition}

\begin{definition}[Deficit and duplicates]
For a column set $T\subseteq \mathbb{F}_2^n$, define its \emph{rank deficit}
\[
\mathrm{def}(T)=|T|-\mathrm{rank}(T)\in\{0,1,\dots\}.
\]
Define $\mathrm{dup}(T)$ as the number of repeated vector values in $T$ (parallel elements), i.e.\ the number of elements whose value has already appeared earlier in the list representation.
Both quantities are $0$ if $T$ is a partial independent set with distinct vectors.
\end{definition}

\begin{definition}[Moves]
Two move types are used.
\begin{itemize}
    \item \emph{Insert:} choose an empty slot $(i,j)$ with $A_{ij}=-1$ and an unused index $k$ in row $i$; set $A_{ij}\leftarrow k$.
    \item \emph{Repair (row-swap exchange):} choose a dependent \emph{source} column $j_s$, a \emph{target} column $j_t\ne j_s$, and a row $i$ with $A_{ij_s}\ne -1$; swap row-$i$ assignments between the two columns:
\[
(A_{ij_s},A_{ij_t}) \leftarrow (A_{ij_t},A_{ij_s}).
\]
\end{itemize}
In the experiments, repair moves are only proposed for rows $i$ corresponding to elements of a circuit returned by the oracle on $T_{j_s}(A)$.

\end{definition}

When a column is dependent, a circuit oracle returns a minimal dependent subset $C\subseteq T_j(A)$ and the \emph{circuit size} $|C|$.
In the fixed-rank experiments ($n=5,7$) the oracle brute-forces minimal dependent subsets among $\le n$ vectors; in the variable-rank benchmark, the oracle searches only for small circuits of size $2,3,4$ (for speed) and otherwise returns the entire dependent set, so the reported circuit size may be an upper bound rather than minimal.

A \emph{policy} is a scoring function $S$ mapping move features to a real score.
At each step, the engine enumerates legal moves, computes post-move features, evaluates $S$, and applies a move with maximum score (ties broken arbitrarily).
Some policies add small random noise to the score, making the algorithm a randomized greedy procedure.

\subsection{Experiment A: Rank 5 circuit-rich pool benchmark}
\label{sec:expA}

\subsubsection{Evaluation harness and instance distribution}

This benchmark sets $n=5$ and works over $\mathbb{F}_2$.
The instance distribution is intentionally \emph{circuit rich}: each basis is sampled from a small pool with many short dependencies.

A pool $P\subset \mathbb{F}_2^5\setminus\{0\}$ is generated from a basis $b_1,\dots,b_5$ by adding structured sums:
\[
P=\{b_i\}_{i=1}^5 \cup \{b_i+b_{i+1}\}_{i=1}^4 \cup \{b_i+b_{i+2}\}_{i=1}^3,
\]
deduplicated and padded if necessary to reach size $12$.

Two variants appear in the user's experiments:
\begin{enumerate}
    \item \emph{Fixed pool:} $b_i=e_i$ is the standard basis, producing a fixed 12-vector pool. 
    \item \emph{Randomized pool:} $b_i$ are columns of a random invertible $5\times 5$ matrix over $\mathbb{F}_2$, i.e.\ a random element of $\mathrm{GL}(5,2)$, producing a fresh pool per instance seed.
\end{enumerate}

A \emph{generic} instance samples each row basis $B_i$ as a random independent $5$-subset of $P$.
A \emph{structured trap} instance first samples a dependent 5-set $D\subset P$, then biases each row basis to contain a forced element from $D$ while maintaining row-independence.

The rollout fitness has the form
\[
\mathrm{fit}=
\mathbf{1}[\text{success}]
+\beta_{\mathrm{full}}
-\gamma_{\mathrm{break}}
-\delta_{\mathrm{repair}}
+\varepsilon\frac{\#\text{valid columns at end}}{n}
-\alpha\frac{\text{steps}}{\text{step limit}},
\]
with $(\alpha,\beta,\gamma,\delta,\varepsilon)=(0.10,0.06,0.12,0.02,0.03)$.

The individual terms were chosen so that the score reads naturally to a combinatorialist.
The indicator $\mathbf{1}[\text{success}]$ is the true objective: producing $n$ disjoint transversal bases.
The auxiliary count $c_{\mathrm{full}}$ rewards reaching a state in which many columns are already valid transversal bases, even if the final few columns still fail.
The counters $c_{\mathrm{break}}$ and $c_{\mathrm{repair}}$ penalize thrashing behaviour (destroying a good column and repeatedly attempting repairs without net progress), and the final normalized valid columns term provides a smooth gradient when success is rare.
Finally, the step penalty discourages policies that succeed only by wandering for a long time.
The robust evaluation over several random seeds is a practical way to suppress brittle policies that exploit accidental ordering effects in the rollout.

In the robust variant, each instance is evaluated over three rollouts with different RNG seeds and counted successful if at least two succeed.

\begin{table}[h]
\centering
\caption{Experiment A (rank 5) harness parameters.}
\begin{tabular}{@{}ll@{}}
\toprule
Parameter & Value \\
\midrule
Rank $n$ & $5$ \\
Pool size & $12$ \\
Step limit & $200$ (robust variant) \\
Instances & $200$ \\
Rollouts/instance & $3$ (majority success) \\
Fitness weights $(\alpha,\beta,\gamma,\delta,\varepsilon)$ & $(0.10,0.06,0.12,0.02,0.03)$ \\
\bottomrule
\end{tabular}
\end{table}

\subsubsection{Results and analysis}

We have run two experiments, each for rank 5, with multiple iteration. 

In the first experiment we worked with a fixed pool (ID 0\_9751). We achieved average score $\approx 1.06$, with fitness\_generic $\approx 1.11$ and fitness\_structured $\approx 1.07$, and overall success rate $1.00$.

In the second experiment we worked with randomized pools with hold-out seeds (ID 0\_9774).
Here we found fitness\_generic $\approx 1.0810$, fitness\_structured $\approx 1.0481$, overall success rate $1.00$ (under majority-vote over three rollouts).

The randomized-pool, multi-rollout evaluation is a substantial robustness check: it disrupts any policy that relies on fixed coordinate artifacts, and it penalizes policies that succeed only with low probability due to score noise.
The slight drop in mean fitness between fixed-pool and randomized-pool runs is consistent with a harder generalization setting: the policy still solves all instances but may require more repairs/steps when the pool is resampled.

\subsubsection{Mathematical algorithm extracted from the evolved policy}

Let $\mathsf{Policy}_5$ denote the evolved rank-5 scoring rule (Appendix~\ref{app:policy-r5}).
At a high level, it implements a lexicographically weighted objective:
At a high level, it implements a lexicographically weighted objective. It begins with \emph{hard vetoes}, forbidding inserts that reduce the number of valid columns or that break a full column. Subject to these vetoes it adds strong \emph{progress incentives}, heavily rewarding moves that increase the number of valid columns and especially moves that complete a full column. When a dependent column exists, it activates \emph{trap-resolution} logic that uses circuit information and deficit/duplicate counts to prioritize repairs, strongly penalizing moves that leave small circuits unresolved. Finally, it uses \emph{controlled noise} only in late stages or when trapped, to help escape local optima without sacrificing determinism on easy instances.

\begin{algorithm}[h]
\caption{Greedy local-exchange construction (as used in all experiments)}
\begin{algorithmic}[1]
\State Initialize assignment matrix $A\leftarrow -1$.
\For{$t=1$ to step limit}
  \If{$\#\{\text{full columns}\}=n$} \State \textbf{return success} \EndIf
  \State Enumerate legal insert moves and legal repair moves.
  \ForAll{legal moves $m$}
    \State Compute post-move features $\phi(A,m)$.
    \State Score $s(m)\leftarrow S(\phi(A,m))$.
  \EndFor
  \State Choose $m^\star \in \arg\max_m s(m)$ and apply $m^\star$ to $A$.
\EndFor
\State \textbf{return failure}
\end{algorithmic}
\end{algorithm}

\begin{conjecture}[Rank-5 toy RBC for circuit-rich pools]
\label{conj:toy5}
On the rank-5 circuit-rich instance class generated from the pool template above (including trap instances, and including arbitrary conjugation by $\mathrm{GL}(5,2)$), the de-randomized version of $\mathsf{Policy}_5$ succeeds: Algorithm~1 terminates with five full columns (five disjoint transversal bases).
\end{conjecture}

\subsection{Experiment B: Rank 7 circuit-rich pool benchmark}
\label{sec:expB}

\subsubsection{Evaluation harness}

This experiment sets $n=7$ and uses an analogous pool template in GF$(2)^7$, producing a pool of size about $2n+2=16$.
The evaluation uses $\texttt{MAX\_STEPS}=350$, $\texttt{NUM\_INSTANCES}=100$, and $\texttt{ROLLOUTS\_PER\_INSTANCE}=3$, with the same fitness weights $(0.10,0.06,0.12,0.02,0.03)$.

\begin{table}[h]
\centering
\caption{Experiment B (rank 7) harness parameters.}
\begin{tabular}{@{}ll@{}}
\toprule
Parameter & Value \\
\midrule
Rank $n$ & $7$ \\
Pool size & $\approx 16$ \\
Step limit & $350$ \\
Instances & $100$ \\
Rollouts/instance & $3$ \\
Fitness weights $(\alpha,\beta,\gamma,\delta,\varepsilon)$ & $(0.10,0.06,0.12,0.02,0.03)$ \\
\bottomrule
\end{tabular}
\end{table}

\subsubsection{Reported results and analysis (ID 0\_7176)}

Reported metrics:
\[
\text{fitness\_generic}\approx 1.05,\qquad
\text{fitness\_structured}\approx 0.88,\qquad
\text{overall\_success\_rate}=1.00,\qquad
\text{average\_score}\approx 0.98.
\]

\noindent\textbf{Interpretation.}
The policy solves all tested instances, but structured fitness drops, suggesting that trap instances at rank 7 require more repairs, more temporary validity breaks, and/or more steps.
This is plausible given that:
(i) each column has more possible near-dependencies, and
(ii) each repair move potentially propagates a constraint across two columns.

\subsubsection{Mathematical algorithm extracted from the evolved rank-7 policy}

The rank-7 policy (Appendix~\ref{app:policy-r7}) retains the same core structure but replaces explicit circuit-tiering by an \emph{adaptive critical threshold} $\tau$ depending on progress and a "desperation" indicator.
It also introduces two nonstandard exploration devices:
(i) additive score noise in late game, and (ii) a deterministic "state signature" congruence bias ("numerology") to perturb tie-breaking in mid-to-late stages.
Mathematically, these are best viewed as randomized/biased tie-breaking mechanisms in the greedy step.

\paragraph{Explicit score decomposition.}
Introduce the shorthand variables (computed from features)
\[
U = 1 + 5.8\left(\frac{\mathrm{full}}{7}\right)^{2.5},\qquad
T = 1 + 4.0\left(\frac{\mathrm{full}}{7}\right)^{2.2},
\]
and the \emph{critical circuit threshold}
\[
\tau = 3 + \mathbf{1}[\mathrm{full}\ge 4] + \mathbf{1}[\mathrm{full}\ge 6] + \mathbf{1}[\text{desperate endgame}].
\]
Let $c$ be the (oracle-provided) circuit size after the move, and let
$\mathrm{crit}=\mathbf{1}[c>0\text{ and }c\le \tau]$.
The policy assigns $-\infty$ to any insert that breaks a full column or creates a critical trap.
Otherwise, it builds an additive score in several pieces. It begins with a large progress term $120{,}000\cdot \Delta\mathrm{valid}$ (scaled up in the endgame). It then adds move-specific rewards/penalties (notably $+300{,}000$ if a repair eliminates a trap, and large penalties if a repair leaves a critical trap). It includes a quality penalty of the form
\[
S_{\mathrm{quality}} = T\cdot\bigl(2000\mathrm{def}(T_{\mathrm{target}})+1500\mathrm{def}(T_{\mathrm{source}})+3500\mathrm{dup}(T_{\mathrm{target}})\bigr),
\]. Finally, it adds exploration perturbations (late-game noise, and a deterministic congruence signature bonus/penalty).

Finally, the entire score is multiplied by $U$ (favoring decisive moves in late progress states).

\begin{conjecture}[Rank-7 toy RBC and de-randomization]
\label{conj:toy7}
On the rank-7 circuit-rich instance class generated from the pool template above (including trap instances and $\mathrm{GL}(7,2)$ conjugation), the evolved policy succeeds with probability $1$ over its internal randomness.
Moreover, there exists a deterministic de-randomization (removing noise and replacing the state-signature heuristic by deterministic tie-breaking) that succeeds on the same class.
\end{conjecture}

\subsection{Experiment C: Variable rank benchmark (ranks 5-13)}
\label{sec:expC}

\subsubsection{Evaluation harness}

This benchmark mixes ranks $n\in\{5,7,9,11,13\}$.
For each instance, a pool of size $\lceil 2.2n\rceil$ is created by the same recipe (basis vectors and low-complexity sums), then $n$ row-bases are sampled (generic vs.\ trap).
The step limit is $\texttt{step\_limit}=20n$.
The evaluator uses $\texttt{NUM\_INSTANCES}=60$ and $\texttt{ROLLOUTS\_PER\_INSTANCE}=1$.
The penalty for breaking validity is increased to $\gamma=0.20$ in this benchmark.

\begin{table}[h]
\centering
\caption{Experiment C (variable rank) harness parameters.}
\begin{tabular}{@{}ll@{}}
\toprule
Parameter & Value \\
\midrule
Ranks & $\{5,7,9,11,13\}$ \\
Pool size & $\lceil 2.2n\rceil$ \\
Step limit & $20n$ \\
Instances & $60$ \\
Rollouts/instance & $1$ \\
Fitness weights $(\alpha,\beta,\gamma,\delta,\varepsilon)$ & $(0.10,0.06,0.20,0.02,0.03)$ \\
\bottomrule
\end{tabular}
\end{table}

The final metrics are:
\[
\text{fitness\_generic}\approx 0.87,\qquad
\text{fitness\_structured}\approx 0.81,\qquad
\text{overall\_success\_rate}=1.00,\qquad
\text{average\_score}\approx 0.89.
\]

Relative fitness decreases compared to fixed-rank experiments because:
(i) one policy must generalize across five ranks, (ii) the step budget is tight ($20n$), and (iii) break-valid events are more heavily penalized ($\gamma=0.20$).

\subsubsection{Mathematical algorithm extracted from the evolved scale-aware policy}

The variable-rank policy (Appendix~\ref{app:policy-var}) defines scale-dependent quantities: \emph{trap energy} $E(C)\propto 1/(|C|-1)$ (higher for smaller circuits); \emph{pressure} $P$ increasing with progress and "columns remaining", with an exponential endgame component; \emph{barrier porosity} for breaking full columns during repairs inside deep traps; and a \emph{singularity avoidance penalty} discouraging moves that create high deficit and high duplicates in the presence of a circuit.

Conceptually, these define a piecewise potential function that trades off (a) making columns full, (b) keeping columns independent, and (c) reducing structural degeneracy, while allowing carefully controlled sacrifices when doing so sharply decreases trap complexity.

\paragraph{Explicit scaling definitions.}
Let $n=\texttt{rank}$, $\rho=\texttt{progress\_ratio}=\mathrm{full}/n$, and $s=n/5$.
Let $c$ be the (oracle-provided) circuit size after the candidate move.
Define the \emph{trap energy}
\[
E(c)=
\begin{cases}
0, & c=0,\\[2mm]
\displaystyle \frac{20}{\max(c,1.5)-1}s, & c>0,
\end{cases}
\qquad
I=\min\!\left(1,\frac{E(c)}{10s}\right).
\]
A "pressure" parameter $P\ge 1$ is computed from $\rho$ and the number of columns remaining $n-\mathrm{full}$, with a linear mid-game ramp and an exponential endgame surge (see code in Appendix~\ref{app:policy-var}).
The policy then defines a \emph{precision multiplier} $\Pi$ scaling penalties on deficits and duplicates, increasing with $P$ and with $\rho$, a \emph{barrier} $B$ for breaking a full column, $B\approx 100{,}000s$ but made \emph{porous} (reduced) during repairs in high trap intensity states when the post-move structure improves, and a \emph{singularity penalty} activated when a move creates simultaneously positive deficit, positive duplicates, and a circuit.

Insert moves are vetoed if they break a full column or decrease validity.
Repairs receive a dominant reward when they remove a circuit ($c=0$) and are penalized otherwise by an amount proportional to $E(c)$ and $P$.
A small adaptive noise term of magnitude $\approx 3500s$ is added, increasing with trap/stagnation and decreasing with pressure.

\begin{conjecture}[Scale-invariant toy RBC on circuit-rich pools]
\label{conj:scaling}
For $n\in\{5,7,9,11,13\}$ and instances generated as above with pool size $\lceil 2.2n\rceil$, the evolved scale-aware policy succeeds within $O(n^2)$ steps.
More ambitiously, there exists a constant $c>2$ such that for all $n\ge 5$, on pool size $\lceil cn\rceil$ generated by the same recipe, a deterministic de-randomization of the policy succeeds.
\end{conjecture}

\appendix 
\section{Graph reconstruction}
\label{app:graph}
\subsection{Bipartite graph reconstruction} 
\subsubsection{AI summary of the experiment} We present the Gemini-generated summary of the experiment first, which we find very informative and useful. It provides hints and insights about the attempted methods and their evolution.
\begin{tcolorbox}[
  enhanced,
  breakable,
  title={Gemini-generated summary},
  colback=gray!3,
  colframe=black!30,
  title after break={Gemini-generated summary (continued)}
]
The problem involves reconstructing a hidden bipartite graph's adjacency matrix by analyzing a "deck" of its vertex-deleted subgraphs, requiring the algorithm to deduce the original topology from the structural invariants preserved in these partial views.

\noindent \textbf{Global Statistics and Degree Constraints}

The initial approach focused on recovering the most basic graph properties using Kelly’s Lemma, a fundamental theorem in graph reconstruction that allows for the calculation of the original graph's total edge count and vertex degree sequence from the deck of subgraphs. Programs like $0\_1$ and $0\_7$ used these recovered global statistics to attempt reconstruction via greedy bipartite realization algorithms (such as variants of Havel-Hakimi or Gale-Ryser). These methods failed to produce high scores because they only satisfied the degree constraints, resulting in graphs that were "degree-equivalent" to the target but topologically distinct (isomorphic ambiguity).

\noindent \textbf{Local Connectivity via Degree Classes}

A significant breakthrough occurred with Program $0\_128$, which shifted the strategy from simply satisfying global degree sequences to inferring local connectivity constraints. Instead of treating all vertices with the same degree as interchangeable, this algorithm compared the degree distributions within specific cards against the global sequence to deduce the exact number of edges connecting a vertex to specific degree-based equivalence classes (e.g., "how many connections does this vertex have to vertices of degree 5?"). This effectively partitioned the reconstruction problem into smaller, constrained blocks, drastically reducing the search space and increasing the score from ~1.15 to ~85.9.

\noindent \textbf{Weisfeiler-Lehman Refinement and Signature Flow}

The next major leap in performance came with Program $0\_220$, which introduced Weisfeiler-Lehman (WL) signature refinement. While the previous approach grouped vertices only by their immediate degree, WL refinement iteratively classifies vertices based on the structural properties of their neighbors, creating highly specific isomorphism-invariant signatures. Subsequent programs (e.g., $0\_336$, $0\_621$) utilized these signatures to model the graph reconstruction as a flow-based system of equations. By calculating the "flow" of signatures between the global graph and the local cards, the algorithms could mathematically deduce the precise number of edges required between different structural vertex types, pushing the score over 330.

\noindent \textbf{Spectral Fingerprinting and 2-Card Voting}

Program $0\_2458$ introduced a new paradigm by incorporating Spectral Graph Theory alongside structural refinement. This approach utilized Singular Value Decomposition (SVD) to generate rotation-invariant spectral signatures for subgraphs. The algorithm inferred the existence of "2-cards" (subgraphs with two vertices removed) and used a voting mechanism based on these spectral signatures to assign confidence scores to individual edges. This allowed the system to move beyond deterministic deduction into probabilistic reconstruction, where edge placement was guided by the consensus of structural evidence across the deck.

\noindent \textbf{Reconstruction as Min-Cost Max-Flow Optimization}

The final evolution in strategy, exemplified by programs like $0\_7568$ and the top-performing $0\_17333$, abandoned greedy heuristics in favor of global optimization. These algorithms framed the matrix reconstruction as a Minimum-Cost Maximum-Flow problem. By assigning "costs" based on the compatibility of WL-derived neighbor profiles and spectral fingerprints, the algorithms could solve for the entire adjacency matrix simultaneously. This ensured that the final graph not only satisfied global degree constraints but also minimized the structural "disagreement" between the reconstructed vertices and the empirical evidence found in the card deck.
\end{tcolorbox}

\subsubsection{The evolved algorithm} The following is the final evolved reconstruction function for $(u,v)$-bipartite graphs. It implements type-based constraints and min-cost flow selection using two-deletion WL-style compatibility signatures.

\begin{lstlisting}[language=Python,basicstyle=\ttfamily\small]
# EVOLVE-BLOCK-START
import numpy as np
from collections import Counter, defaultdict
import networkx as nx

MCMF_COST_SCALE = 1_000_000

def _compute_node_features(
    degrees: np.ndarray,
    n_sums: np.ndarray,
    n_sq_sums: np.ndarray,
    nn_sums: np.ndarray,
    nn_sq_sums: np.ndarray,
    nnn_sums: np.ndarray,
    n_distinct_deg_counts: np.ndarray
) -> list[tuple]:
    return list(zip(
        degrees.tolist(), n_sums.tolist(), n_sq_sums.tolist(),
        nn_sums.tolist(), nn_sq_sums.tolist(), nnn_sums.tolist(),
        n_distinct_deg_counts.tolist()
    ))

def get_neighbor_profile_of_removed_vertex(global_degs_counts: Counter, card_degs: np.ndarray) -> tuple:
    c_counts = Counter(card_degs)
    neighbors = []
    all_degs_keys = set(global_degs_counts.keys()) | set(c_counts.keys())
    if not all_degs_keys: return ()
    max_deg = max(all_degs_keys)
    current_shift = 0 
    for k in range(max_deg, -1, -1): 
        x_k = global_degs_counts[k] + current_shift - c_counts[k]
        if x_k > 0: neighbors.extend([k] * x_k)
        current_shift = x_k
    return tuple(sorted(neighbors))

def remove_from_profile(profile: tuple, val: int) -> tuple:
    p_list = list(profile)
    if val in p_list: p_list.remove(val)
    return tuple(p_list)

def get_matrix_profiles(matrix: np.ndarray, axis: int) -> list:
    if axis == 0:
        other_degs = np.sum(matrix, axis=1).astype(int)
        return [tuple(sorted(other_degs[matrix[:, j] == 1].tolist())) for j in range(matrix.shape[1])]
    else:
        other_degs = np.sum(matrix, axis=0).astype(int)
        return [tuple(sorted(other_degs[matrix[i, :] == 1].tolist())) for i in range(matrix.shape[0])]

def get_matrix_signature_components(matrix: np.ndarray):
    if matrix.size == 0:
        return (np.array([], dtype=np.int64),) * 14
    
    r_d = np.sum(matrix, axis=1, dtype=np.int64)
    c_d = np.sum(matrix, axis=0, dtype=np.int64)
    
    r_n_sum = matrix.dot(c_d)
    r_n_sq_sum = matrix.dot(c_d ** 2)
    c_n_sum = matrix.T.dot(r_d)
    c_n_sq_sum = matrix.T.dot(r_d ** 2)

    r_nn_sum = matrix.dot(c_n_sum)
    r_nn_sq_sum = matrix.dot(c_n_sq_sum)
    c_nn_sum = matrix.T.dot(r_n_sum)
    c_nn_sq_sum = matrix.T.dot(r_n_sq_sum)

    r_nnn_sum = matrix.dot(c_nn_sum)
    c_nnn_sum = matrix.T.dot(r_nn_sum)

    r_n_distinct = np.zeros(matrix.shape[0], dtype=np.int64)
    for i in range(matrix.shape[0]):
        nbs = np.nonzero(matrix[i])[0]
        if len(nbs) > 0: r_n_distinct[i] = len(set(c_d[j] for j in nbs))

    c_n_distinct = np.zeros(matrix.shape[1], dtype=np.int64)
    for j in range(matrix.shape[1]):
        nbs = np.nonzero(matrix[:, j])[0]
        if len(nbs) > 0: c_n_distinct[j] = len(set(r_d[i] for i in nbs))
    
    return r_d, r_n_sum, r_n_sq_sum, r_nn_sum, r_nn_sq_sum, r_nnn_sum, r_n_distinct, \
           c_d, c_n_sum, c_n_sq_sum, c_nn_sum, c_nn_sq_sum, c_nnn_sum, c_n_distinct

def _cosine_sim(t1: tuple, t2: tuple) -> float:
    if not t1 and not t2: return 1.0
    if not t1 or not t2: return 0.0
    c1, c2 = Counter(t1), Counter(t2)
    common = c1.keys() & c2.keys()
    dot = sum(c1[k] * c2[k] for k in common)
    n1 = sum(v**2 for v in c1.values())
    n2 = sum(v**2 for v in c2.values())
    return dot / (np.sqrt(n1) * np.sqrt(n2)) if n1 and n2 else 0.0

def get_2card_signature(matrix: np.ndarray, wl_steps: int = 3) -> tuple:
    rows, cols = matrix.shape
    if rows == 0 or cols == 0: return ((), (), matrix.shape)

    comps = get_matrix_signature_components(matrix)
    r_colors = _compute_node_features(*comps[:7])
    c_colors = _compute_node_features(*comps[7:])

    for _ in range(wl_steps):
        next_r = [0] * rows
        for i in range(rows):
            nbs = tuple(sorted(c_colors[j] for j in np.nonzero(matrix[i])[0]))
            next_r[i] = hash((r_colors[i], nbs))
        
        next_c = [0] * cols
        for j in range(cols):
            nbs = tuple(sorted(r_colors[i] for i in np.nonzero(matrix[:, j])[0]))
            next_c[j] = hash((c_colors[j], nbs))
        
        r_colors, c_colors = next_r, next_c

    return (tuple(sorted(r_colors)), tuple(sorted(c_colors)), matrix.shape)

def reconstruct(deck: list[np.ndarray], u: int, v: int) -> np.ndarray:
    if u == 0 or v == 0: return np.zeros((u, v), dtype=int)
    if u == 1 and v == 1: return np.zeros((1, 1), dtype=int)

    sum_card_edges = sum(int(np.sum(c)) for c in deck)
    m = sum_card_edges // (u + v - 2)

    u_indices = [i for i, c in enumerate(deck) if c.shape == (u - 1, v)]
    v_indices = [i for i, c in enumerate(deck) if c.shape == (u, v - 1)]
    if len(u_indices) != u or len(v_indices) != v: return np.zeros((u, v), dtype=int)

    d_u = [m - int(np.sum(deck[i])) for i in u_indices]
    d_v = [m - int(np.sum(deck[j])) for j in v_indices]
    d_v_counts, d_u_counts = Counter(d_v), Counter(d_u)

    def get_types(indices, degrees, global_opp_counts, is_u_card):
        types = []
        for i, idx in enumerate(indices):
            card = deck[idx]
            axis = 0 if is_u_card else 1
            local_degs = np.sum(card, axis=axis).astype(np.int64)
            prof = get_neighbor_profile_of_removed_vertex(global_opp_counts, local_degs)
            types.append((degrees[i], prof))
        return types

    u_types = get_types(u_indices, d_u, d_v_counts, True)
    v_types = get_types(v_indices, d_v, d_u_counts, False)
    u_type_counts, v_type_counts = Counter(u_types), Counter(v_types)
    dist_u_types = sorted(list(u_type_counts.keys()), reverse=True)
    dist_v_types = sorted(list(v_type_counts.keys()), reverse=True)

    def compute_reqs(indices, degrees, dist_opp_types, opp_type_counts, is_u_card):
        reqs_list = []
        for i, idx in enumerate(indices):
            card = deck[idx]
            axis = 0 if is_u_card else 1
            prof_local = get_matrix_profiles(card, axis=axis)
            d_local = np.sum(card, axis=axis).astype(int)
            observed = Counter(zip(d_local, prof_local))
            
            reqs = {}
            shifted_in = Counter()
            remover_deg = degrees[i]
            
            for T in dist_opp_types:
                d, prof = T
                needed = opp_type_counts[T] - (observed[T] - shifted_in[T])
                if needed > 0:
                    reqs[T] = needed
                    shifted_in[(d - 1, remove_from_profile(prof, remover_deg))] += needed
            reqs_list.append(reqs)
        return reqs_list

    u_reqs = compute_reqs(u_indices, d_u, dist_v_types, v_type_counts, True)
    v_reqs = compute_reqs(v_indices, d_v, dist_u_types, u_type_counts, False)

    u_2card_sigs, v_2card_sigs = [], []
    for i, idx in enumerate(u_indices):
        card = deck[idx]
        local_v_types = list(zip(np.sum(card, axis=0).astype(int), get_matrix_profiles(card, axis=0)))
        sub_sigs = defaultdict(list)
        for c_idx in range(v):
            sub_sigs[local_v_types[c_idx]].append(get_2card_signature(np.delete(card, c_idx, axis=1)))
        u_2card_sigs.append({k: tuple(sorted(v)) for k, v in sub_sigs.items()})

    for j, idx in enumerate(v_indices):
        card = deck[idx]
        local_u_types = list(zip(np.sum(card, axis=1).astype(int), get_matrix_profiles(card, axis=1)))
        sub_sigs = defaultdict(list)
        for r_idx in range(u):
            sub_sigs[local_u_types[r_idx]].append(get_2card_signature(np.delete(card, r_idx, axis=0)))
        v_2card_sigs.append({k: tuple(sorted(v)) for k, v in sub_sigs.items()})

    all_card_sigs = [get_2card_signature(c) for c in deck]
    sim_matrix = np.zeros((u + v, u + v))
    for i in range(u + v):
        sim_matrix[i, i] = 1.0
        for j in range(i + 1, u + v):
            s1, s2 = all_card_sigs[i], all_card_sigs[j]
            sim = (_cosine_sim(s1[0], s2[0]) + _cosine_sim(s1[1], s2[1])) / 2.0
            sim_matrix[i, j] = sim_matrix[j, i] = sim

    u_sort_keys = {i: (all_card_sigs[u_indices[i]], tuple(np.sort(sim_matrix[u_indices[i], :]))) for i in range(u)}
    v_sort_keys = {j: (all_card_sigs[v_indices[j]], tuple(np.sort(sim_matrix[v_indices[j], :]))) for j in range(v)}
    u_ranks = {idx: r for r, idx in enumerate(sorted(range(u), key=lambda i: u_sort_keys[i]))}
    v_ranks = {idx: r for r, idx in enumerate(sorted(range(v), key=lambda j: v_sort_keys[j]))}

    rows_by_type = defaultdict(list)
    for i, t in enumerate(u_types): rows_by_type[t].append(i)
    for t in rows_by_type: rows_by_type[t].sort(key=lambda x: u_sort_keys[x])

    cols_by_type = defaultdict(list)
    for j, t in enumerate(v_types): cols_by_type[t].append(j)
    for t in cols_by_type: cols_by_type[t].sort(key=lambda x: v_sort_keys[x])

    adj = np.zeros((u, v), dtype=int)
    for r_type in dist_u_types:
        dr, prof_r = r_type
        for c_type in dist_v_types:
            dc, prof_c = c_type
            r_idxs, c_idxs = rows_by_type[r_type], cols_by_type[c_type]
            if not r_idxs or not c_idxs: continue

            target_v_conn = (dc - 1, remove_from_profile(prof_c, dr))
            target_u_conn = (dr - 1, remove_from_profile(prof_r, dc))
            target_v_non = c_type
            target_u_non = r_type

            G_flow = nx.DiGraph()
            source, sink = "S", "T"
            candidates = []

            for r_i in r_idxs:
                cap = u_reqs[r_i].get(c_type, 0)
                if cap <= 0: continue
                G_flow.add_edge(source, f"R_{r_i}", capacity=cap, weight=0)
                
                u_conn = Counter(u_2card_sigs[r_i].get(target_v_conn, ()))
                u_non = Counter(u_2card_sigs[r_i].get(target_v_non, ()))
                
                for c_j in c_idxs:
                    if v_reqs[c_j].get(r_type, 0) <= 0: continue
                    v_conn = Counter(v_2card_sigs[c_j].get(target_u_conn, ()))
                    v_non = Counter(v_2card_sigs[c_j].get(target_u_non, ()))
                    
                    sc_conn = (2.0 * sum((u_conn & v_conn).values()) / (sum(u_conn.values()) + sum(v_conn.values()))) if (u_conn or v_conn) else 0
                    sc_non = (2.0 * sum((u_non & v_non).values()) / (sum(u_non.values()) + sum(v_non.values()))) if (u_non or v_non) else 0
                    candidates.append((sc_conn - sc_non, r_i, c_j))

            for c_j in c_idxs:
                cap = v_reqs[c_j].get(r_type, 0)
                if cap > 0: G_flow.add_edge(f"C_{c_j}", sink, capacity=cap, weight=0)
            
            candidates.sort(key=lambda x: (-x[0], u_ranks[x[1]], v_ranks[x[2]]))
            for conf, r_i, c_j in candidates:
                cost = int(-conf * MCMF_COST_SCALE) - (u_ranks[r_i] * v + v_ranks[c_j])
                G_flow.add_edge(f"R_{r_i}", f"C_{c_j}", capacity=1, weight=cost)
            
            try:
                flow = nx.max_flow_min_cost(G_flow, source, sink)
                for r_i in r_idxs:
                    for c_j in c_idxs:
                        if flow.get(f"R_{r_i}", {}).get(f"C_{c_j}", 0) == 1:
                            adj[r_i, c_j] = 1
            except nx.NetworkXError:
                pass
    return adj
# EVOLVE-BLOCK-END
\end{lstlisting}

\subsection{Planar evolved block}
\noindent The following is the final representative evolved planar reconstruction function achieving perfect reconstruction under the hardened planar harness (exact deck equality verification).

\begin{lstlisting}[language=Python,basicstyle=\ttfamily\small]
# EVOLVE-BLOCK-START
import numpy as np
import networkx as nx
import math
from collections import Counter
from itertools import combinations, product

def reconstruct(deck: list[np.ndarray], n: int) -> np.ndarray:
    # 1. Kelly's Lemma for edge count
    sum_card_edges = sum(np.sum(card) // 2 for card in deck)
    if n > 2:
        m = sum_card_edges // (n - 2)
    elif n == 2:
        return np.array([[0, 1], [1, 0]], dtype=int)
    else:
        return np.zeros((n, n), dtype=int)

    # 2. Derive Degree Sequence of G
    card_edges = [np.sum(card) // 2 for card in deck]
    removed_degrees = [m - ce for ce in card_edges]
    G_deg_counts = Counter(removed_degrees)

    candidates = []

    # 3. Analyze each card to find valid re-attachment configurations
    for i, card in enumerate(deck):
        deg_v = removed_degrees[i]
        C_degrees = np.sum(card, axis=1).astype(int)
        C_deg_counts = Counter(C_degrees)

        target_rem = G_deg_counts.copy()
        target_rem[deg_v] -= 1
        if target_rem[deg_v] == 0:
            del target_rem[deg_v]

        req_conn = {}
        carry = 0
        possible = True

        all_degs = set(C_deg_counts.keys()) | set(target_rem.keys())
        if not all_degs:
            sorted_degs = []
        else:
            sorted_degs = range(min(all_degs), max(all_degs) + 2)

        for x in sorted_degs:
            n_c = C_deg_counts[x]
            n_g = target_rem[x]
            next_carry = carry + n_c - n_g
            if next_carry < 0 or next_carry > n_c:
                possible = False
                break
            if next_carry > 0:
                req_conn[x] = next_carry
            carry = next_carry

        if possible and carry == 0:
            if sum(req_conn.values()) == deg_v:
                ambiguity = 1
                for x, count in req_conn.items():
                    ambiguity *= math.comb(C_deg_counts[x], count)
                candidates.append({
                    'ambiguity': ambiguity,
                    'card_idx': i,
                    'req': req_conn,
                    'card_degs': C_degrees
                })

    candidates.sort(key=lambda x: x['ambiguity'])

    # 5. Try to reconstruct from best candidates
    for cand in candidates:
        if cand['ambiguity'] > 2000:
            continue

        card_idx = cand['card_idx']
        req = cand['req']
        card = deck[card_idx]
        C_degrees = cand['card_degs']

        deg_to_indices = {}
        for idx, d in enumerate(C_degrees):
            deg_to_indices.setdefault(d, []).append(idx)

        choice_pools = []
        for d, count in req.items():
            pool = list(combinations(deg_to_indices[d], count))
            choice_pools.append(pool)

        for p in product(*choice_pools):
            neighbors = [idx for group in p for idx in group]

            A_hat = np.zeros((n, n), dtype=int)
            A_hat[:-1, :-1] = card
            A_hat[neighbors, n-1] = 1
            A_hat[n-1, neighbors] = 1

            # exact_deck_equal and get_deck_from_adj are provided by harness
            if exact_deck_equal(get_deck_from_adj(A_hat), deck):
                return A_hat

    return nx.to_numpy_array(nx.gnm_random_graph(n, m), dtype=int)
# EVOLVE-BLOCK-END
\end{lstlisting}

\section{Alon-Tarsi conjecture}
\label{app:alon}

\subsection{Experiment 1: evolving a nontrivial sign-flipping involution with biased residual set}

\subsubsection{AI summary of the experiment}
The following is the short description of the evolution process generated by Gemini

\begin{tcolorbox}[
  enhanced,
  breakable,
  title={Gemini-generated summary},
  colback=gray!3,
  colframe=black!30,
  title after break={Gemini-generated summary (continued)}
]
The problem is to create a deterministic algorithm that transforms any Latin square of even order into a different valid Latin square with the opposite mathematical sign (parity), such that applying the transformation twice returns the original square.

\noindent \textbf{Deterministic Transpositions}

The initial approach, represented by 0\_1, attempted to solve the problem by partitioning the input space using hashes and applying different deterministic strategies based on the input's hash. While this ensured the involution property (the function is its own inverse), the sign-flipping capability was inconsistent. The strategy improved significantly with 0\_5, which abandoned hash-based randomness in favor of size-specific elementary swaps. By applying a fixed swap of rows or columns depending on the order nn, the program leveraged the algebraic property that an odd permutation of rows negates the sign of a Latin square.

\noindent \textbf{Local Cycle Decompositions}

A major mathematical breakthrough occurred with 0\_53, which moved beyond simple row/column swaps to internal structural manipulation. This algorithm introduced odd-cycle trades, identifying odd-length cycles in the permutation defined by two rows and swapping symbols only along those cycles. This exploits the insight that swapping values along a single odd cycle acts as a localized transposition, guaranteeing a sign flip while preserving the Latin square constraints. Subsequent iterations like 0\_62 and 0\_134 refined this by targeting disjoint row pairs and ensuring the total number of performed trades remained odd, thereby strictly enforcing the sign-reversal requirement.

\noindent \textbf{Global Symmetries and Conjugation}

Program 0\_175 introduced a different class of transformation based on global algebraic symmetry, specifically row-symbol conjugation. By mapping the Latin square triple (r,c,v)(r,c,v) to (r,v,c)(r,v,c), the algorithm utilized the geometric properties of the square (parastrophy) rather than local swaps. 0\_324 further refined this by employing row-symbol swapping (r,c,v)→(v,c,r)(r,c,v)→(v,c,r). These approaches proved highly effective because conjugation acts as a structural involution that, for specific orders of nn, consistently reverses the sign of the square.

\noindent \textbf{Conflict Resolution for Invertibility}

As the algorithms began attempting more complex searches for cycle trades, ensuring the map remained a perfect involution (applying it twice returns exactly the original input) became difficult. 0\_275 addressed this by introducing backtracking conflict verification. This mechanism ensures that a high-priority rule is only applied if it does not interfere with the ability of the reverse transformation to identify the same rule, essentially checking that the move is a locally stable maximum in the transition graph. 0\_796 expanded this into a prioritized search with recursive checks, allowing the system to explore diverse row and column matchings to find a valid sign-reversing move that is guaranteed to be reversible.

\noindent \textbf{Hybrid Strategies and Modulo Arithmetic}

A significant leap in performance came with 0\_817, which recognized that the mathematical properties of Latin squares differ depending on their order modulo 4. This program implemented a bifurcated strategy: it applies global conjugation for orders $n \equiv 2\ (mod\ 4)$ and a prioritized search for odd-cycle trades for orders $n \equiv 0\ (mod\ 4)$. This hybrid approach acknowledges that while conjugation is a powerful tool, it does not universally flip signs for all even orders, necessitating different algebraic techniques for different dimensions.

\noindent \textbf{Optimizing Residual Bias via "Sabotage"}

The final evolution in strategy focused on the "Residual Bias" metric—optimizing the score even when a perfect sign flip was not found. 0\_946 and 0\_1026 introduced the concept of sign-preserving "sabotage". When the algorithm encounters a difficult case where a sign-reversing move cannot be cleanly resolved, it intentionally performs an even-cycle trade or a specific symbol swap. While this fails the sign-flip objective, it forces the failure into a known parity state, maximizing the squared sign bias of the residual set. Programs 0\_2052 and 0\_9094 perfected this by combining cycle decomposition, priority-based conflict checking, and deterministic sabotage, creating a highly robust system that flips signs whenever possible and fails gracefully and consistently when not.
\end{tcolorbox}

\subsubsection{The evolved algorithm} The final evolve  block (run 9094 after 13.6 hours running) is the following.
\medskip
\begin{lstlisting}[language=Python]
# EVOLVE-BLOCK-START
# The following code is the part that AlphaEvolve can modify.
# Its goal is to implement a map F: L -> L' on Latin squares.
# Ideally, F should be a sign-reversing involution for most inputs L.

import numpy as np

def _find_trade(S, n, idx1, idx2, mode, strategy):
    if mode == 0:  # Row mode
        elem1, elem2 = S[idx1], S[idx2]
        elem2_pos = np.argsort(elem2)
        perm = elem2_pos[elem1]
    else:  # Column mode
        elem1, elem2 = S[:, idx1], S[:, idx2]
        elem2_pos = np.argsort(elem2)
        perm = elem2_pos[elem1]

    visited = np.zeros(n, dtype=bool)
    odd_cycles, even_cycles = [], []
    for i in range(n):
        if not visited[i]:
            cycle = []
            curr = i
            while not visited[curr]:
                visited[curr] = True
                cycle.append(curr)
                curr = perm[curr]
            if len(cycle) > 1:
                (odd_cycles if len(cycle) % 2 == 1 else even_cycles).append(cycle)

    odd_cycles.sort(key=min)
    even_cycles.sort(key=min)

    trade_indices = []
    if strategy == "flip":
        if not odd_cycles:
            return None
        trade_indices = odd_cycles[0][:]
        if even_cycles:
            trade_indices.extend(even_cycles[0])
    elif strategy == "sabotage":
        # Goal: Sign-preserving trade (Even cycle)
        if not even_cycles:
            return None
        candidate_cycle = even_cycles[0]
        # Deterministic rarity check (accept approx 5%)
        h_val = sum(x * (x + 13) for x in candidate_cycle)
        if h_val % 20 != 0:
            return None
        trade_indices = candidate_cycle[:]
    else:
        return None

    S_new = S.copy()
    if mode == 0:
        S_new[idx1, trade_indices], S_new[idx2, trade_indices] = \
            S_new[idx2, trade_indices], S_new[idx1, trade_indices]
    else:
        S_new[trade_indices, idx1], S_new[trade_indices, idx2] = \
            S_new[trade_indices, idx2], S_new[trade_indices, idx1]
    return S_new

def parity_switching_map(L: list[list[int]], n: int) -> list[list[int]]:
    """
    Applies a sign-reversing involution to the Latin square L.
    Strategies:
    1. Search for a 'flip' trade (odd cycle + optional even cycle) in rows/cols.
    2. Search for a 'sabotage' trade (even cycle) with rarity constraints.
    3. Fallback to a conditional strategy for residual bias.
    """
    L_np = np.array(L)
    half_n = n // 2

    ROW_MODE = 0
    COL_MODE = 1

    rng = np.random.RandomState(42) # Deterministic random state
    
    # Flattened list of all candidate trades to check in order.
    # Each entry is (mode, idx1, idx2, strategy)
    all_candidate_trades_sequence = []

    # Configuration for different orders
    use_sabotage = (n > 8)
    num_random_iters = 50

    # 0. Sabotage Phase (Row mode) - Sign-preserving trades
    if use_sabotage:
        sab_perm = rng.permutation(n)
        sab_pairs_to_check = [(sab_perm[0], sab_perm[1])]
        if half_n >= 2:
             sab_pairs_to_check.append((sab_perm[2], sab_perm[3]))

        for idx1, idx2 in sab_pairs_to_check:
            all_candidate_trades_sequence.append((ROW_MODE, idx1, idx2, "sabotage"))

    # 1. Row Matchings (Flip strategy)
    # Standard stride-based matching
    for i in range(half_n):
        all_candidate_trades_sequence.append((ROW_MODE, i, i + half_n, "flip"))

    # Random matchings
    for _ in range(num_random_iters):
        perm = rng.permutation(n)
        for k in range(half_n):
            all_candidate_trades_sequence.append((ROW_MODE, perm[2*k], perm[2*k+1], "flip"))

    # 2. Column Matchings (Flip strategy)
    # Standard stride-based matching
    for i in range(half_n):
        all_candidate_trades_sequence.append((COL_MODE, i, i + half_n, "flip"))

    # Random matchings
    for _ in range(num_random_iters):
        perm = rng.permutation(n)
        for k in range(half_n):
            all_candidate_trades_sequence.append((COL_MODE, perm[2*k], perm[2*k+1], "flip"))

    # Function to check for prior conflicts
    def _has_prior_conflict(S, current_trade_index):
        for i in range(current_trade_index):
            mode, idx1, idx2, strategy = all_candidate_trades_sequence[i]
            if _find_trade(S, n, idx1, idx2, mode, strategy) is not None:
                return True
        return False

    # Main search for trades
    for current_trade_index, (mode, idx1, idx2, strategy) in enumerate(all_candidate_trades_sequence):
        L_tent_result = _find_trade(L_np, n, idx1, idx2, mode, strategy)
        
        if L_tent_result is not None:
            # A trade was found. Check if it conflicts with any earlier potential trade in the sequence.
            if not _has_prior_conflict(L_tent_result, current_trade_index):
                return L_tent_result.tolist()

    # Fallback: Conditional strategy to ensure residual bias.
    # The default 'conjugate_row_symbol' operation (L_conj[rows, L_np] = cols)
    # is sign-flipping if n/2 is odd (n=10, 14, 26).
    # To ensure a biased residual set, we need a sign-preserving fallback in such cases.
    # A symbol swap (0<->1) is always a sign-preserving involution for even n > 1.
    if half_n % 2 == 1: # For n=10, 14, 26, 'conjugate_row_symbol' flips sign.
        # Use a sign-preserving symbol swap as fallback.
        L_fallback = L_np.copy()
        mask_0, mask_1 = (L_fallback == 0), (L_fallback == 1)
        L_fallback[mask_0] = 1
        L_fallback[mask_1] = 0
        return L_fallback.tolist()
    else: # For n=8, 12, 'conjugate_row_symbol' preserves sign.
        # Use 'conjugate_row_symbol' as fallback.
        L_conj = np.zeros_like(L_np)
        rows, cols = np.indices((n, n))
        # This performs row-symbol conjugation: L_conj[r, k] = c if L_np[r, c] = k.
        L_conj[rows, L_np] = cols
        return L_conj.tolist()

# EVOLVE-BLOCK-END
\end{lstlisting}

\subsection*{Experiment 2 EVOLVE block}
\begin{lstlisting}[language=Python]
# ============================================================
# EVOLVE-BLOCK-START
# ============================================================
import numpy as np

def _get_odd_cycles(arr1, arr2, n):
    """Helper: Find all odd cycles (length > 1) in the permutation between two arrays."""
    inv2 = np.argsort(arr2)
    perm = inv2[arr1]
    
    visited = np.zeros(n, dtype=bool)
    cycles = []
    
    for i in range(n):
        if visited[i]:
            continue
        c = []
        cur = i
        while not visited[cur]:
            visited[cur] = True
            c.append(cur)
            cur = perm[cur]
        if len(c) > 1 and len(c) % 2 == 1:
            cycles.append(c)
    return cycles

def _get_1_factors(size):
    """
    Generates a 1-factorization of K_size (size must be even).
    Returns a list of lists of pairs. Each inner list is a perfect matching.
    Construction based on the standard 'pattern' method for K_2m.
    Vertices are 0..2m-1.
    """
    factors = []
    # For K_n, vertices are 0, ..., n-1. We fix vertex n-1 (the "infinity" vertex).
    # The other n-1 vertices (0, ..., n-2) are arranged in a regular (n-1)-gon.
    m = size - 1 # This is the modulo base for the (n-1)-gon vertices.
    inf_vertex = size - 1
    
    for turn in range(m):
        pairs = []
        # Edge connected to the "infinity" vertex
        pairs.append(tuple(sorted((turn, inf_vertex))))
        
        # Other edges (between vertices in the (n-1)-gon)
        for k in range(1, (size // 2)): # k goes from 1 to n/2 - 1
            # Vertex a and b are symmetric around 'turn' in the (n-1)-gon
            a = (turn - k + m) % m # +m to handle negative results of modulo correctly in Python
            b = (turn + k) % m
            pairs.append(tuple(sorted((a, b))))
        
        factors.append(pairs)
    return factors

def _find_best_trade_in_conjugate_space(M: np.ndarray, n: int, pairs_to_check: list[tuple[int, int]]) -> tuple | None:
    """
    Helper to find the best odd cycle trade in a given matrix M,
    considering specified pairs of indices.

    Returns:
        tuple: (key, idx1, idx2, cycle_elems) or None
    """
    best_trade_info = None # Stores (key, idx1, idx2, cycle_elems)

    for pair_idx_ordinal, (idx1, idx2) in enumerate(pairs_to_check):
        arr1 = M[idx1]
        arr2 = M[idx2]

        odd_cycles = _get_odd_cycles(arr1, arr2, n)

        for cyc in odd_cycles:
            # Key for minimization: (length (asc), pair_idx_ordinal (asc), sorted_cycle (asc))
            key = (len(cyc), pair_idx_ordinal, tuple(sorted(cyc)))

            if best_trade_info is None or key < best_trade_info[0]:
                best_trade_info = (key, idx1, idx2, cyc)

    return best_trade_info

def parity_switching_map(L: list[list[int]], n: int) -> list[list[int]]:
    """
    Multi-Perspective Local Trade Search across rows, columns, and symbols.
    Deterministic, attempts to apply a smallest odd-cycle trade.
    Fallback: swap rows 0 and 1.
    """
    S = np.array(L, dtype=int)

    # -- Define pairing strategies --
    def get_adjacent_pairs(size):
        return [(i, i + 1) for i in range(0, size - 1, 2)]

    def get_half_shifted_pairs(size):
        if size % 2 != 0:
            return []
        return [(i, i + size // 2) for i in range(size // 2)]

    def get_all_pairs(size):
        pairs = []
        for i in range(size):
            for j in range(i + 1, size):
                pairs.append((i, j))
        return pairs

    pairing_strategies = {
        'adjacent': get_adjacent_pairs,
        'half_shifted': get_half_shifted_pairs,
        'one_factor_flat': lambda size: sum(_get_1_factors(size), []),
        'exhaustive': get_all_pairs
    }

    ordered_strategies = [
        'adjacent',
        'half_shifted',
        'one_factor_flat',
        'exhaustive'
    ]

    conjugate_types = ['row', 'col', 'sym']

    for conjugate_type in conjugate_types:
        for strategy_name in ordered_strategies:
            get_pairs_func = pairing_strategies[strategy_name]

            current_M = S
            if conjugate_type == 'col':
                current_M = S.T
            elif conjugate_type == 'sym':
                # S_sym where S_sym[v, c] = r (symbol v is at row r in column c)
                current_M = np.argsort(S, axis=0)

            pairs_to_check = get_pairs_func(n)

            trade_info = _find_best_trade_in_conjugate_space(current_M, n, pairs_to_check)

            if trade_info:
                _, idx1, idx2, cycle_elems = trade_info
                T = S.copy()

                if conjugate_type == 'row':
                    T[idx1, cycle_elems], T[idx2, cycle_elems] = T[idx2, cycle_elems], T[idx1, cycle_elems]
                elif conjugate_type == 'col':
                    T[cycle_elems, idx1], T[cycle_elems, idx2] = T[cycle_elems, idx2], T[cycle_elems, idx1]
                elif conjugate_type == 'sym':
                    for c_idx in cycle_elems:
                        r1 = current_M[idx1, c_idx]
                        r2 = current_M[idx2, c_idx]
                        T[r1, c_idx] = idx2
                        T[r2, c_idx] = idx1
                return T.tolist()

    # Fallback
    T = S.copy()
    T[[0, 1], :] = T[[1, 0], :]
    return T.tolist()

# ============================================================
# EVOLVE-BLOCK-END
# ============================================================
\end{lstlisting}

\subsection*{Experiment 3 EVOLVE block}
\begin{lstlisting}[language=Python]
# ============================================================
# EVOLVE-BLOCK-START
# ============================================================
import numpy as np

# -- Helper: Apply Symbol Swap --
def _apply_symbol_swap(matrix, s1, s2):
    T = matrix.copy()
    mask_s1 = (matrix == s1)
    mask_s2 = (matrix == s2)
    T[mask_s1] = s2
    T[mask_s2] = s1
    return T

# -- Helper: Fixed Row/Col Swaps (for meta-trades) --
def _fixed_row_swap(matrix_np, r1, r2):
    M_copy = matrix_np.copy()
    M_copy[[r1, r2]] = M_copy[[r2, r1]]
    return M_copy

def _fixed_col_swap(matrix_np, c1, c2):
    M_copy = matrix_np.copy()
    M_copy[:, [c1, c2]] = M_copy[:, [c2, c1]]
    return M_copy

# -- Helper: Canonical Cycle --
def _canonical_cycle(cyc: list[int]) -> list[int]:
    if not cyc: return []
    m = min(cyc)
    k = cyc.index(m)
    return cyc[k:] + cyc[:k]

def _stabilized_canonical_cycle(cyc: list[int]) -> list[int]:
    """
    To ensure F(F(L))=L, the choice of cycle must be stable
    under reversal. We use the lexicographically smaller of a
    cycle and its reverse as the canonical representation.
    """
    if not cyc: return []
    c1 = _canonical_cycle(cyc)
    if len(c1) <= 2:
        return c1
    rev_cyc = c1[:1] + c1[1:][::-1]
    c2 = _canonical_cycle(rev_cyc)
    return min(c1, c2)

# -- Helper: Get all odd cycles between two lines --
def _get_odd_cycles(matrix, idx1, idx2, n: int, axis_is_col=False):
    if axis_is_col:
        line1, line2 = matrix[:, idx1], matrix[:, idx2]
    else:
        line1, line2 = matrix[idx1], matrix[idx2]

    inv2 = np.full(n, -1, dtype=int)
    inv2[line2] = np.arange(n, dtype=int)
    perm = inv2[line1] 
    
    visited = np.zeros(n, dtype=bool)
    cycles = []

    for i in range(n):
        if visited[i]: continue
        path_trace = []
        cur = i
        while not visited[cur]:
            visited[cur] = True
            path_trace.append(int(cur))
            cur = int(perm[cur])
        
        if cur in path_trace:
            start_idx = path_trace.index(cur)
            cyc = path_trace[start_idx:]
            if len(cyc) > 1 and (len(cyc) % 2 == 1):
                cycles.append(_stabilized_canonical_cycle(cyc))
    
    cycles.sort(key=lambda x: (len(x), x))
    return cycles

# Helper: Collect all canonical trades
def _collect_all_trades(M, n: int):
    trades = []
    for r1 in range(n):
        for r2 in range(r1 + 1, n):
            cycles = _get_odd_cycles(M, r1, r2, n, axis_is_col=False)
            for cyc in cycles:
                key = (len(cyc), cyc, r1, r2, 0) # 0 for row trade
                data = (cyc, False, r1, r2)
                trades.append((key, data))
    for c1 in range(n):
        for c2 in range(c1 + 1, n):
            cycles = _get_odd_cycles(M, c1, c2, n, axis_is_col=True)
            for cyc in cycles:
                key = (len(cyc), cyc, c1, c2, 1) # 1 for col trade
                data = (cyc, True, c1, c2)
                trades.append((key, data))
    
    trades.sort(key=lambda x: x[0])
    return trades

# -- Helper: Check for higher-priority trades --
def _has_tier1_trade(M_np, n: int):
    for i in range(n // 2):
        if _get_odd_cycles(M_np, 2 * i, 2 * i + 1, n, axis_is_col=False):
            return True
    return False

def _has_safe_tier2_trade(M_np, n: int):
    for i in range(n // 2):
        c1, c2 = 2 * i, 2 * i + 1
        cycles = _get_odd_cycles(M_np, c1, c2, n, axis_is_col=True)
        for cyc in cycles:
            T_prime = M_np.copy()
            ind = np.array(cyc, dtype=int)
            T_prime[ind, c1], T_prime[ind, c2] = T_prime[ind, c2], T_prime[ind, c1]
            if not _has_tier1_trade(T_prime, n):
                return True
    return False

# Helper: Check stability (No Tier 1 or Safe Tier 2 trades)
def _is_stable_matrix(M, n: int):
    return (not _has_tier1_trade(M, n)) and (not _has_safe_tier2_trade(M, n))

# -- NEW HELPER FUNCTION TO ENCAPSULATE TRADE FINDING LOGIC --
def _apply_first_odd_cycle_trade_if_stable(M_np, n: int):
    # 1. Search for Row Trades in adjacent pairs
    for i in range(n // 2):
        r1, r2 = 2 * i, 2 * i + 1
        cycles = _get_odd_cycles(M_np, r1, r2, n, axis_is_col=False)
        if cycles:
            cycle = cycles[0]
            T = M_np.copy()
            indices_to_swap = np.array(cycle, dtype=int)
            T[r1, indices_to_swap], T[r2, indices_to_swap] = T[r2, indices_to_swap], T[r1, indices_to_swap]
            return T

    # 2. Search for Column Trades in adjacent pairs
    for i in range(n // 2):
        c1, c2 = 2 * i, 2 * i + 1
        cycles = _get_odd_cycles(M_np, c1, c2, n, axis_is_col=True)
        for cycle in cycles:
            T = M_np.copy()
            indices_to_swap = np.array(cycle, dtype=int)
            T[indices_to_swap, c1], T[indices_to_swap, c2] = T[indices_to_swap, c2], T[indices_to_swap, c1]
            
            if not _has_tier1_trade(T, n):
                return T
    
    # 3. Tier 3: Exhaustive Global Search with Canonical Verification.
    all_candidates = _collect_all_trades(M_np, n)
    
    for cand_key, cand_data in all_candidates:
        cycle_to_apply, axis_is_col, idx1, idx2 = cand_data
        
        T = M_np.copy()
        indices = np.array(cycle_to_apply, dtype=int)
        if axis_is_col:
            T[indices, idx1], T[indices, idx2] = T[indices, idx2], T[indices, idx1]
        else:
            T[idx1, indices], T[idx2, indices] = T[idx2, indices], T[idx1, indices]
            
        # 1. Stability Check
        if not _is_stable_matrix(T, n):
            continue
            
        # 2. Canonical Verification
        is_canonical = True
        trades_in_T = _collect_all_trades(T, n)
        for t_key, t_data in trades_in_T:
            if t_key > cand_key:
                break
            if t_key == cand_key:
                break
            
            c_p, ax_p, i1_p, i2_p = t_data
            T_prime = T.copy()
            ind_p = np.array(c_p, dtype=int)
            if ax_p:
                T_prime[ind_p, i1_p], T_prime[ind_p, i2_p] = T_prime[ind_p, i2_p], T_prime[ind_p, i1_p]
            else:
                T_prime[i1_p, ind_p], T_prime[i2_p, ind_p] = T_prime[i2_p, ind_p], T_prime[i1_p, ind_p]
            
            if _is_stable_matrix(T_prime, n):
                is_canonical = False
                break
        
        if is_canonical:
            return T
            
    return None

def parity_switching_map(L: list[list[int]], n: int) -> list[list[int]]:
    """
    Deterministic involution on Latin squares.
    Searches for the first available 'odd cycle trade' in pairs of rows or columns.
    """
    S = np.array(L, dtype=int)
    result_S = _apply_first_odd_cycle_trade_if_stable(S, n)
    if result_S is not None:
        return result_S.tolist()

    # CANONICAL ORDER TIER 1: Row Swaps
    for r1 in range(n):
        for r2 in range(r1 + 1, n):
            S_transformed = _fixed_row_swap(S, r1, r2)
            result_transformed = _apply_first_odd_cycle_trade_if_stable(S_transformed, n)

            if result_transformed is not None:
                T_final = _fixed_row_swap(result_transformed, r1, r2)
                if _apply_first_odd_cycle_trade_if_stable(T_final, n) is None:
                    return T_final.tolist()

    # CANONICAL ORDER TIER 2: Column Swaps
    for c1 in range(n):
        for c2 in range(c1 + 1, n):
            S_transformed = _fixed_col_swap(S, c1, c2)
            result_transformed = _apply_first_odd_cycle_trade_if_stable(S_transformed, n)

            if result_transformed is not None:
                T_final = _fixed_col_swap(result_transformed, c1, c2)

                if _apply_first_odd_cycle_trade_if_stable(T_final, n) is not None:
                    continue

                has_row_swap_solution = False
                for r1_check in range(n):
                    for r2_check in range(r1_check + 1, n):
                        T_final_transformed = _fixed_row_swap(T_final, r1_check, r2_check)
                        if _apply_first_odd_cycle_trade_if_stable(T_final_transformed, n) is not None:
                            has_row_swap_solution = True
                            break
                    if has_row_swap_solution:
                        break
                
                if not has_row_swap_solution:
                    return T_final.tolist()

    # Fallback
    if n > 1:
        return _apply_symbol_swap(S, 0, 1).tolist()
        
    return S.tolist()
# ============================================================
# EVOLVE-BLOCK-END
# ============================================================
\end{lstlisting}

\section{Rota conjecture}
\label{app:rota}

\subsection{Experiment A policy: rank-5 trap-avoidance and circuit-breaking}
\label{app:policy-r5}

The evolved evolved policy (evolve blocks) is the following.

\medskip
\begin{lstlisting}[language=Python]
# EVOLVE-BLOCK-START
# AlphaEvolve may modify ONLY this block.
# Goal: evolve a strong, interpretable policy scoring function.
def policy_function(features: dict[str, float]) -> float:
    """
    A hybrid policy balancing robust trap avoidance with a highly focused repair strategy.

    Key Principles:
    1.  **Strict Insert Hygiene:** `Insert` moves are for safe progress. Any insert that
        breaks a full column or creates a small, critical trap is vetoed.
    2.  **High-Stakes Repairs:** `Repair` moves are for solving traps. They operate on a
        high-risk, high-reward basis. Escaping a trap yields a massive bonus, while
        breaking a full column incurs a massive penalty. A repair is only chosen if
        the expected reward (trap escape) vastly outweighs the risk (breaking progress).
    3.  **Dynamic Risk Assessment:** A "despair" metric allows the agent to be more
        aggressive with repairs late in the game, but the fundamental cost/benefit
        trade-offs for risky moves remain stark. This ensures that even desperate
        actions are directed towards high-potential solutions.
    """
    # Termination check
    if features.get("is_terminate_move", 0.0) > 0.5:
        return 1e9 if features.get("global_num_full_cols", 0.0) >= 5.0 else -1e9

    # Import for stochastic element during "panic" and exponential calculations for 'girth singularity'.
    import random
    import math

    score = 0.0

    # Feature extraction
    num_full = features.get("global_num_full_cols", 0.0)
    delta_valid = features.get("delta_num_valid", 0.0)
    becomes_full = features.get("target_col_becomes_full", 0.0) > 0.5
    breaks_full = features.get("source_col_was_full_and_is_not_anymore", 0.0) > 0.5
    is_repair = features.get("is_repair_move", 0.0) > 0.5
    is_insert = features.get("is_insert_move", 0.0) > 0.5

    target_deficit = features.get("target_rank_deficit_after", 0.0)
    source_deficit = features.get("source_rank_deficit_after", 0.0)
    circuit_size = features.get("circuit_size", 0.0)
    target_dups = features.get("target_dup_count_after", 0.0)

    # -- Adaptive Policy Modulators --
    # Trap Severity Index: Assigns discrete severity levels to circuits based on their size.
    trap_severity = 0.0
    trap_penalty_amplifier = 1.0

    if circuit_size > 0.0:
        if circuit_size <= 2.0:
            trap_severity = 10.0
            trap_penalty_amplifier = 2.5
        elif circuit_size <= 3.0:
            trap_severity = 4.0
            trap_penalty_amplifier = 1.5
        elif circuit_size <= 4.0:
            trap_severity = 1.5
            trap_penalty_amplifier = 1.2
        elif circuit_size <= 5.0:
            trap_severity = 0.5
            trap_penalty_amplifier = 1.05

    is_critical_trap_after = (trap_severity >= 1.5)
    has_trap = trap_severity > 0.0

    tension_factor = 1.0
    if has_trap:
        messiness = target_deficit + source_deficit + target_dups
        tension_factor = 1.0 + (trap_severity * 0.08) * (1.0 + messiness)
        tension_factor = min(tension_factor, 7.0)

    PHASE_EXPANSION = (num_full < 3.0)
    PHASE_STABILIZATION = (num_full >= 3.0 and num_full < 4.0)
    PHASE_ENDGAME = (num_full >= 4.0)

    urgency_factor = 1.0 + (num_full * 0.4)

    despair_index = 0.0
    if has_trap:
        despair_index += trap_severity * 0.5
        if is_critical_trap_after:
            despair_index += 1.0
        if is_repair:
            despair_index += 1.5 * trap_severity
            if is_critical_trap_after:
                despair_index += 0.5

    if delta_valid < 0:
        despair_index += 2.5
    elif (delta_valid == 0 and not becomes_full) and (PHASE_ENDGAME or has_trap):
        despair_index += 1.5

    despair_index = min(despair_index, 7.0)

    if is_insert:
        is_insert_vetoed = (
            breaks_full or
            is_critical_trap_after or
            (PHASE_EXPANSION and has_trap) or
            (has_trap and not becomes_full)
        )
        if is_insert_vetoed:
            return -1e9

        if has_trap: 
            score -= 80000.0 * trap_severity * trap_penalty_amplifier * urgency_factor

        if not (becomes_full and delta_valid > 0):
            unproductive_insert_penalty = 1500.0 * urgency_factor
            unproductive_insert_penalty *= tension_factor
            unproductive_insert_penalty *= (1.0 + despair_index * 0.2)
            score -= unproductive_insert_penalty

    elif is_repair:
        base_repair_penalty = 3500.0 * urgency_factor

        if has_trap:
            raw_desperation_discount = base_repair_penalty * (despair_index * 0.5)
            max_possible_discount = base_repair_penalty + 2000.0 * urgency_factor
            actual_desperation_discount = min(raw_desperation_discount, max_possible_discount)
            base_repair_penalty -= actual_desperation_discount

        score -= base_repair_penalty

        trap_was_resolved = not has_trap

        if source_deficit > 0 and not trap_was_resolved:
            score -= 1500.0 * source_deficit * urgency_factor

        if is_critical_trap_after:
            penalty = 180000.0 * trap_penalty_amplifier * urgency_factor
            if PHASE_STABILIZATION or PHASE_ENDGAME:
                penalty *= 2.5 
            score -= penalty
        elif has_trap:
            penalty_multiplier = 1.0
            if PHASE_STABILIZATION:
                penalty_multiplier = 2.0

            score -= trap_severity * 4500.0 * penalty_multiplier * trap_penalty_amplifier * urgency_factor
        else:
            trap_escape_reward = 300000.0 * urgency_factor
            if PHASE_STABILIZATION:
                trap_escape_reward *= 1.4
            score += trap_escape_reward

            if target_dups == 0:
                clean_repair_bonus = 100000.0 * urgency_factor * (1.0 + despair_index * 0.7)
                if PHASE_STABILIZATION or PHASE_ENDGAME:
                    clean_repair_bonus *= 1.6
                score += clean_repair_bonus

        is_endgame_gambit = PHASE_ENDGAME and becomes_full
        is_circuit_breaker = (PHASE_STABILIZATION or despair_index > 1.0) and trap_was_resolved

        if breaks_full and (is_endgame_gambit or is_circuit_breaker):
            base_sacrifice_bonus = 160000.0 * urgency_factor
            quality_multiplier = 1.0 if target_deficit == 0 else 0.8
            sacrifice_bonus = base_sacrifice_bonus * quality_multiplier
            sacrifice_bonus *= (1.0 + despair_index * 0.5) 
            score += sacrifice_bonus

        elif breaks_full:
            score -= 130000.0 * urgency_factor

    phase_bonus_multiplier = 1.0
    if PHASE_EXPANSION:
        phase_bonus_multiplier = 1.5
    elif PHASE_STABILIZATION:
        phase_bonus_multiplier = 1.2

    if target_deficit == 0:
        score += 3500.0 * urgency_factor * phase_bonus_multiplier
    if target_dups == 0:
        score += 3000.0 * urgency_factor * phase_bonus_multiplier

    if delta_valid > 0:
        progress_reward = 120000.0 * urgency_factor
        if is_repair:
            progress_reward *= (1.5 + despair_index * 0.5)
        if PHASE_ENDGAME:
            progress_reward += 30000.0 * urgency_factor
        score += progress_reward
    elif delta_valid < 0:
        regression_penalty = 100000.0 * urgency_factor
        if is_repair:
            regression_penalty *= (1.2 + despair_index * 0.3)
        score -= regression_penalty

    score -= 1500.0 * target_deficit * urgency_factor * tension_factor
    score -= 1000.0 * source_deficit * urgency_factor * tension_factor

    dup_penalty = 3000.0 * target_dups
    if target_dups > 1.0:
        dup_penalty *= 1.5
    score -= dup_penalty * urgency_factor * tension_factor

    if PHASE_ENDGAME or (PHASE_STABILIZATION and is_repair) or despair_index > 1.0:
        base_noise_mag = 400.0 * urgency_factor
        despair_amplification = math.exp(despair_index * 0.35)
        noise_mag = base_noise_mag * despair_amplification
        score += random.uniform(-noise_mag, noise_mag)

    return score
# EVOLVE-BLOCK-END
\end{lstlisting}

\subsection{Experiment B policy: rank-7 endgame pressure and numerology exploration}
\label{app:policy-r7}
The evolved evolved policy (evolve blocks) is the following.
\medskip

\begin{lstlisting}[language=Python]
# EVOLVE-BLOCK-START
# AlphaEvolve may modify ONLY this block.
# Goal: evolve a strong, interpretable policy scoring function for Rank 7.
def policy_function(features: dict[str, float]) -> float:
    """
    A hybrid policy balancing robust trap avoidance with a highly focused repair strategy.
    Adapted for Rank 7 (higher complexity, longer games).
    """
    # Termination check: Threshold updated to 7.0 for Rank 7
    if features.get("is_terminate_move", 0.0) > 0.5:
        return 1e9 if features.get("global_num_full_cols", 0.0) >= 7.0 else -1e9

    import random

    score = 0.0

    # Feature extraction
    num_full = features.get("global_num_full_cols", 0.0)
    delta_valid = features.get("delta_num_valid", 0.0)
    becomes_full = features.get("target_col_becomes_full", 0.0) > 0.5
    breaks_full = features.get("source_col_was_full_and_is_not_anymore", 0.0) > 0.5
    is_repair = features.get("is_repair_move", 0.0) > 0.5
    is_insert = features.get("is_insert_move", 0.0) > 0.5

    target_deficit = features.get("target_rank_deficit_after", 0.0)
    source_deficit = features.get("source_rank_deficit_after", 0.0)
    circuit_size = features.get("circuit_size", 0.0)
    target_dups = features.get("target_dup_count_after", 0.0)

    # -- State & Phase Definitions --
    PHASE_ENDGAME = (num_full >= 6.0)

    is_desperate_endgame_state = (PHASE_ENDGAME and not becomes_full and delta_valid <= 0)

    critical_circuit_threshold = 3.0 + \
                                 (num_full >= 4.0) + \
                                 (num_full >= 6.0) + \
                                 is_desperate_endgame_state

    has_trap_after = circuit_size > 0.0
    is_critical_trap_after = has_trap_after and (circuit_size <= critical_circuit_threshold)

    urgency_factor = 1.0 + (num_full / 7.0)**2.5 * 5.8
    time_pressure_factor = 1.0 + (num_full / 7.0)**2.2 * 4.0

    # 1. Primary Objective: Progress vs. Regression
    if delta_valid > 0:
        reward = 120000.0 * delta_valid
        if PHASE_ENDGAME:
             reward *= 1.5
        score += reward
    elif delta_valid < 0:
        score -= 100000.0 * abs(delta_valid)

    # 2. Move-Specific Incentives
    if is_insert:
        if breaks_full or is_critical_trap_after:
            return -1e9

        if becomes_full:
            score += 80000.0
            if target_deficit == 0 and target_dups == 0:
                score += 40000.0
        elif delta_valid == 0: 
            if target_deficit <= 1.0 and target_dups == 0.0 and not has_trap_after:
                 score += 25000.0 * (1.0 + num_full / 7.0)

            stalling_penalty_base = 2000.0
            if has_trap_after:
                stalling_penalty_base += 15000.0
                if is_critical_trap_after:
                    stalling_penalty_base += 30000.0
                stalling_penalty_base += 5000.0 * (target_deficit + target_dups)

            if is_desperate_endgame_state:
                is_messy_stall = has_trap_after or target_deficit > 1.0 or target_dups > 0
                if is_messy_stall:
                    stalling_penalty_base += 40000.0
                else:
                    stalling_penalty_base += random.uniform(6000.0, 22000.0)
                    if random.random() < 0.6:
                        score += random.uniform(7000.0, 18000.0)

            endgame_multiplier = 2.0 if PHASE_ENDGAME else 1.0
            stalling_penalty = stalling_penalty_base * endgame_multiplier * time_pressure_factor
            score -= stalling_penalty

    elif is_repair:
        trap_was_resolved_by_this_move = not has_trap_after
        if trap_was_resolved_by_this_move:
            reward = 300000.0
            if PHASE_ENDGAME:
                reward *= 2.0 
            score += reward

        trap_penalty = 0.0
        if is_critical_trap_after:
            trap_penalty = 200000.0
            if PHASE_ENDGAME:
                trap_penalty *= 3.0
        elif has_trap_after: 
            trap_penalty = 50000.0
            if PHASE_ENDGAME:
                trap_penalty *= 1.5
        score -= trap_penalty

        if breaks_full:
            penalty = 150000.0
            if trap_was_resolved_by_this_move:
                penalty = -25000.0
            elif is_desperate_endgame_state:
                 is_clean_after = not has_trap_after and target_deficit == 0.0 and source_deficit == 0.0
                 if is_clean_after:
                     penalty = -15000.0
                 else:
                     penalty *= 0.4
            elif PHASE_ENDGAME:
                penalty *= 1.5
            score -= penalty

        base_repair_cost = 4000.0

        if trap_was_resolved_by_this_move:
            base_repair_cost -= 2500.0
        elif has_trap_after:
            base_repair_cost += 1000.0 * circuit_size

        if target_deficit <= 1.0 and target_dups == 0.0:
            base_repair_cost -= 1500.0

        if PHASE_ENDGAME and (is_critical_trap_after or is_desperate_endgame_state):
            repair_cost_multiplier = 0.3
        else:
            repair_cost_multiplier = time_pressure_factor
        score -= base_repair_cost * repair_cost_multiplier

    # 3. State Quality Penalties
    quality_penalty = 0.0
    quality_penalty += 2000.0 * target_deficit
    quality_penalty += 1500.0 * source_deficit
    quality_penalty += 3500.0 * target_dups
    quality_penalty *= time_pressure_factor
    score -= quality_penalty

    # 4. Exploration
    if num_full >= 5.0:
        score += random.uniform(-4000.0, 4000.0)

    # 5. "Numerology" Heuristic
    state_signature = int(num_full * 1e2) + \
                      int(target_deficit * 1e1) + \
                      int(source_deficit * 7) + \
                      int(target_dups * 5) + \
                      int(circuit_size)

    is_fortunate = (state_signature % 13 == 0)
    is_cursed = (state_signature % 19 == 0)

    if num_full >= 4.0:
        numerology_bonus = 0.0
        if is_fortunate and not is_cursed:
            numerology_bonus = 35000.0
        elif is_cursed and not is_fortunate:
            numerology_bonus = -45000.0
        elif is_cursed and is_fortunate:
            numerology_bonus = -80000.0
        score += numerology_bonus

    score *= urgency_factor
    return score
# EVOLVE-BLOCK-END
\end{lstlisting}

\subsection{Experiment C policy: scale-aware trap energy and pressure dynamics}
\label{app:policy-var}

\subsubsection{AI Summary of the experiment}
The following is the short description of the evolution process generated by Gemini

\begin{tcolorbox}[
  enhanced,
  breakable,
  title={Gemini-generated summary},
  colback=gray!3,
  colframe=black!30,
  title after break={Gemini-generated summary (continued)}
]
The goal of these programs is to evolve a general-purpose constructive policy for Rota's Basis Conjecture that generalizes across variable vector space dimensions (Ranks 5 through 13) without hardcoding dimension-specific logic. The solvers must balance rapid basis construction in low dimensions with high-precision conflict resolution in high dimensions to identify independent transversals and resolve "trap" instances containing hidden dependency circuits.

\noindent \textbf{Phase-Dependent Heuristic Scaling}

The initial successful strategies moved away from static logic, recognizing that different stages of construction require different behaviors. Program 0\_1 introduced phase-dependent heuristic scaling, utilizing the progress\_ratio to increase the magnitude of rewards and penalties as the solution neared completion. This created an "Endgame" phase with higher urgency. Program 0\_15 refined this by adding a rank-relative urgency factor, ensuring that the transition from exploration to precision was scaled appropriate to the dimension (nn) of the problem, preventing the solver from being too aggressive in high-rank spaces where precision is paramount.

\noindent \textbf{Non-Linear Complexity and Chaotic Exploration}

A significant performance jump occurred with the introduction of non-linear adjustments to handle the exponential increase in search space difficulty as rank increases. Program 0\_200 implemented non-linear rank-complexity scaling, allowing the policy to apply drastically different move priorities for Rank 5 versus Rank 13. Simultaneously, to prevent stagnation in local optima (traps), this program introduced a feature-seeded pseudo-chaotic exploration oscillator. This mechanism injected structured noise into the decision-making process, allowing the agent to shake loose from suboptimal configurations.

\noindent \textbf{Strategic Sacrifice and Trap Resolution}

As the solvers improved, the ability to resolve specific "trap" instances (hidden low-girth circuits) became the primary differentiator. Program 0\_276 achieved a major score increase by formalizing sacrificial repair moves. Instead of strictly preserving completed columns, this strategy allowed the solver to break a full column if it resolved a dependency circuit. This was coupled with non-linear trap severity mitigation, which weighted the penalty of a dependency loop inversely to its size, prioritizing the destruction of tight, critical loops over larger, looser dependencies.

\noindent \textbf{Dynamic Policy Cohesion and Bio-Mimetic Regulation}

Mid-tier programs focused on smoother transitions between search states. Program 0\_1034 utilized Dynamic Policy Cohesion (DPC), using a tanh-modulated signal to transition fluidly from expansion to consolidation rather than using hard thresholds. Later, programs like 0\_1882 and 0\_5406 adopted a neurotransmitter-inspired dynamic regulation system. This modeled internal states analogous to "Dopamine" (reward sensitivity) and "Cortisol" (stress/urgency). These variables modulated the dynamic deficit exponent, enforcing increasingly strict mathematical independence constraints as the "stress" of the system (progress and rank) increased.

\noindent \textbf{Trap Energy and Barrier Tunneling}

The final evolutionary leap abandoned simple penalties in favor of physics-inspired landscape navigation. Starting with 0\_10675, the programs began modeling dependency loops as trap energy. The key insight was the implementation of barrier porosity, which allowed the agent to "tunnel" through high-cost barriers. This logic mathematically incentivized heroic sacrifices—intentionally breaking a valid basis configuration to escape a local energy minimum—but only if the resulting reduction in trap energy was significant. The highest performing program, 0\_14464, refined this by combining exponential progress-based pressure with circuit-aware trap mitigation, enabling it to solve the highest rank instances by strategically deconstructing its own progress to bypass deep structural traps.
\end{tcolorbox}

\subsubsection{The evolved algorithm} After 21.5 hours we arrived to the following evolved block (run 14464).
\medskip
\begin{lstlisting}[language=Python]
# EVOLVE-BLOCK-START
# AlphaEvolve may modify ONLY this block.
# Goal: evolve a policy that generalizes across Ranks 5-13.
def policy_function(features: dict[str, float]) -> float:
    """
    A general-purpose policy for Rota's Basis Conjecture.
    Scales logic based on 'rank' and 'progress_ratio'.
    """
    rank = features.get("rank", 5.0)
    num_full = features.get("global_num_full_cols", 0.0)
    progress_ratio = features.get("progress_ratio", 0.0) # 0.0 to 1.0

    # Dynamic Termination: Only terminate if we have ALL columns full
    if features.get("is_terminate_move", 0.0) > 0.5:
        return 1e12 if num_full >= rank else -1e12

    import random
    import math

    score = 0.0

    delta_valid = features.get("delta_num_valid", 0.0)
    becomes_full = features.get("target_col_becomes_full", 0.0) > 0.5
    breaks_full = features.get("source_col_was_full_and_is_not_anymore", 0.0) > 0.5
    is_repair = features.get("is_repair_move", 0.0) > 0.5
    is_insert = features.get("is_insert_move", 0.0) > 0.5

    target_deficit = features.get("target_rank_deficit_after", 0.0)
    source_deficit = features.get("source_rank_deficit_after", 0.0)
    circuit_size = features.get("circuit_size", 0.0)
    target_dups = features.get("target_dup_count_after", 0.0)
    # -- Common Penalty Coefficients --
    C_DEFICIT_BASE = 4000.0
    C_DUP_BASE = 3000.0

    # -- Rank & Pressure Scaling --
    rank_scaler = rank / 5.0

    # -- Trap Physics --
    trap_energy = 0.0
    if circuit_size > 0.0:
        trap_energy = 20.0 / (max(circuit_size, 1.5) - 1.0) * rank_scaler

    trap_intensity = 0.0
    if trap_energy > 0.0:
        trap_intensity = min(1.0, trap_energy / (10.0 * rank_scaler + 1e-6))

    progress_intensity = progress_ratio
    early_midgame_intensity = 1.0 - progress_ratio

    # -- Stagnation Physics: For non-circuit structural issues --
    stagnation_intensity = 0.0
    if circuit_size == 0.0 and (target_deficit > 0.0 or target_dups > 0.0 or (is_repair and source_deficit > 0.0)):
        stagnation_score = target_deficit * 2.0 + target_dups * 2.0
        if is_repair:
            stagnation_score += source_deficit * 1.0
        stagnation_intensity = min(1.0, (stagnation_score / (rank * 2.0 + 1e-6)) * 0.6)

    # -- Pressure Engine --
    pressure = 1.0
    columns_remaining = max(1.0, rank - num_full) 

    exponential_rate = (8.0 * math.sqrt(rank_scaler)) * (1.0 - trap_intensity * 0.6 - stagnation_intensity * 0.3)

    mid_game_start = 0.5
    end_game_start = 0.75 

    if progress_ratio > mid_game_start:
        linear_progress = (progress_ratio - mid_game_start) / (end_game_start - mid_game_start + 1e-6)
        pressure += 0.5 * linear_progress * rank_scaler

    if progress_ratio > end_game_start:
        pressure += math.exp(exponential_rate * (progress_ratio - end_game_start)) - 1.0

    remaining_ratio = columns_remaining / rank
    if progress_ratio > 0.8 or columns_remaining <= 3:
        singularity_imminence_factor = 20.0 * rank_scaler * (1.0 - trap_intensity * 0.7 - stagnation_intensity * 0.3)
        singularity_pressure = math.exp(-singularity_imminence_factor * remaining_ratio)
        pressure += singularity_pressure * max(0.0, 1.0 - trap_intensity - stagnation_intensity * 0.5)

    pressure = max(1.0, pressure * ((1.0 + rank_scaler / 2.0) * (1.0 - trap_intensity * 0.2 - stagnation_intensity * 0.1)))

    # -- Barrier for breaks_full --
    barrier_porosity = 1.0
    if is_repair and trap_energy > 0.0:
        reduction_from_trap = min(0.8, trap_intensity * 0.8)

        improvement_signal = 0.0
        if target_deficit == 0.0 and target_dups == 0.0:
            improvement_signal = 1.0
        elif target_deficit < rank / 5.0 or target_dups < 2.0:
            improvement_signal = 0.5

        reduction_from_improvement = min(0.3, improvement_signal * math.sqrt(rank_scaler) * 0.5)
        reduction_from_urgency_baseline = min(0.4, math.sqrt(rank_scaler) * 0.2)

        total_reduction = reduction_from_trap + reduction_from_improvement + reduction_from_urgency_baseline
        barrier_porosity = 1.0 - min(0.95, total_reduction)

    current_barrier_height = (100000.0 * rank_scaler) * barrier_porosity

    # -- Structural Instability Penalty (SIP) --
    structural_instability_penalty = 0.0
    crisis_signal_for_sip = max(trap_intensity, stagnation_intensity)
    if crisis_signal_for_sip > 0.1:
        degeneracy_term = target_deficit * target_dups
        entanglement_term = 0.0
        if is_repair:
            entanglement_term = 0.5 * source_deficit * target_deficit

        total_instability = degeneracy_term + entanglement_term
        structural_instability_penalty = (100000.0 * rank_scaler * (1.0 + pressure * 0.5)) * total_instability * trap_intensity
        structural_instability_penalty = min(structural_instability_penalty, 750000.0 * rank_scaler)

    # -- Precision Multiplier --
    precision_multiplier = ((1.0 + pressure * 0.5) * rank_scaler) * (
        1.0 + progress_intensity * 1.5 - trap_intensity * 0.75 - stagnation_intensity * 0.5
    )
    precision_multiplier = min(precision_multiplier, 10.0 * rank_scaler * (1.0 + pressure * 0.5))

    # -- Singularity Avoidance Penalty --
    singularity_avoidance_penalty = 0.0
    if target_deficit > 0.0 and target_dups > 0.0 and circuit_size > 0.0:
        base_penalty = 1.5e6 * rank_scaler
        progress_amplification = math.exp(5.0 * progress_ratio) - 1.0
        singularity_avoidance_penalty = base_penalty * (1.0 + progress_amplification) * (1.0 + pressure)
        singularity_avoidance_penalty *= (1.0 - trap_intensity * 0.5)
        singularity_avoidance_penalty = min(singularity_avoidance_penalty, 5.0e6 * rank_scaler * (1.0 + pressure))

    # Heroic sacrifice
    if is_repair and breaks_full and circuit_size == 0.0 and target_dups == 0.0 and target_deficit == 0.0:
        meltdown_crisis_signal = max(trap_intensity, stagnation_intensity * 0.8)
        meltdown_reward_magnitude = 250000.0 * rank_scaler * meltdown_crisis_signal * (1.0 + pressure)
        min_perfect_resolution_reward = 50000.0 * rank_scaler * (1.0 + pressure)
        current_barrier_height = -max(meltdown_reward_magnitude, min_perfect_resolution_reward)

    # Inserts
    if is_insert:
        if breaks_full:
            return -1e9

        if delta_valid > 0:
            score += 80000.0 * delta_valid * (1.0 + pressure + early_midgame_intensity * 0.5 + progress_intensity * 0.25)
        elif delta_valid < 0:
            return -1e9

        if becomes_full:
            score += 200000.0 * (1.0 + pressure + progress_intensity * 0.5)

        if circuit_size > 0.0:
            circuit_creation_penalty = 750000.0 * trap_energy * pressure * (1.0 - early_midgame_intensity * 0.3)
            score -= circuit_creation_penalty

        score -= C_DEFICIT_BASE * math.pow(target_deficit, 1.5) * precision_multiplier
        score -= C_DUP_BASE * math.pow(target_dups, 1.2) * precision_multiplier
        score -= structural_instability_penalty
        score -= singularity_avoidance_penalty

        if delta_valid == 0 and not becomes_full:
             score -= 1000.0 * (1.0 + pressure + progress_intensity * 0.5)

    # Repairs
    elif is_repair:
        score -= 2000.0 * rank_scaler * (1.0 - trap_intensity * 0.5)

        score -= C_DEFICIT_BASE * math.pow(target_deficit, 1.5) * precision_multiplier
        score -= C_DUP_BASE * math.pow(target_dups, 1.2) * precision_multiplier
        score -= C_DEFICIT_BASE * math.pow(source_deficit, 1.5) * precision_multiplier
        score -= structural_instability_penalty
        score -= singularity_avoidance_penalty

        if delta_valid > 0:
            score += 80000.0 * delta_valid * (1.0 + pressure + trap_energy * 0.4)
        elif delta_valid < 0:
            validity_decrease_penalty = 40000.0 * delta_valid * (1.0 + pressure)
            leniency_reduction_factor = 1.0 - (trap_intensity * 0.75 * (1.0 - progress_ratio * 0.5))
            validity_decrease_penalty *= leniency_reduction_factor
            score += validity_decrease_penalty

        score -= current_barrier_height

        if circuit_size == 0.0:
            circuit_break_reward = 800000.0 * rank_scaler * pressure * (1.0 + trap_intensity)
            score += circuit_break_reward
        else:
            failed_repair_penalty = 40000.0 * trap_energy * pressure
            score -= failed_repair_penalty

    if target_deficit == 0:
        score += 2000.0 * rank_scaler * (1.0 + pressure * 0.5 + progress_intensity)
    if target_dups == 0:
        score += 1000.0 * rank_scaler * (1.0 + pressure * 0.5 + progress_intensity)

    combined_frustration = trap_energy + (stagnation_intensity * 5.0 * rank_scaler)
    noise_magnitude = 3500.0 * rank_scaler * min(3.0, combined_frustration / (1.0 + pressure * 0.25 + 1e-6))
    score += random.uniform(-noise_magnitude, noise_magnitude)

    return score
# EVOLVE-BLOCK-END
\end{lstlisting}


\end{document}